%% file: main.tex
\documentclass[onefignum,onetabnum]{siamart251104}

\usepackage{amsmath,amssymb,bbm,todonotes}

\input{shared}
\usepackage{tikz} \usetikzlibrary{arrows.meta,positioning,calc,backgrounds}
\usepackage{pgfplots} \pgfplotsset{compat=1.18}
\usepackage{placeins}

\newcount\mstylenum 
\def\varstyle#1{\mathchoice{\mstylenum=0 #1}{\mstylenum=1 #1}{\mstylenum=2 #1}{\mstylenum=3 #1}}
\def\usestyle{\ifcase\mstylenum \displaystyle\or\textstyle\or\scriptstyle\or\scriptscriptstyle\fi}
\def\gensim#1{\mathrel{\varstyle{\vtop{\offinterlineskip
   \halign{\hfil$\usestyle##$\hfil\cr#1\cr\noalign{\kern1pt}\sim\cr}}}}}

\ifpdf
\hypersetup{
  pdftitle={Hybrid Analog--Digital Approximate Inverse Preconditioning for Krylov Methods},
  pdfauthor={S. Shah, R. Li, T. Gokmen, V. Kalantzis, L. Horesh, and Y. Xi}
}
\fi

\begin{document}

\maketitle
\begin{abstract}
Analog in-memory computing enables highly parallel matrix--vector
multiplications with reduced data movement, but the resulting operations are
noisy, quantized, and affected by device- and circuit-level non-idealities.
This paper studies approximate inverse preconditioning for Krylov subspace
methods in a \edit{hybrid analog--digital architecture that combines a
conventional digital processor with an analog in-memory computing (AIMC)
unit}. The digital host performs sparse
products with the coefficient matrix and the precision-sensitive Krylov
operations, while the AIMC unit performs preconditioner applications through
analog crossbar matrix--vector multiplications. Since the realized preconditioner is
inexact and application-dependent, the outer iteration is formulated as the
flexible GMRES method. 
We show that analog execution changes the usual preconditioner design problem in the sense that 
a stronger digital
preconditioner may be less effective after analog scaling, write noise, input/output perturbations, quantization, and clipping are taken into account. 
\edit{We use exact block inverses and sparse approximate
inverses as reference constructions and develop a
noise-matching criterion for Monte Carlo approximate
inverse (MCAI) construction. We also study damping and nesting
to improve robustness.} Numerical experiments, \edit{in which analog matrix--vector products are simulated rather than executed on hardware}, show that analog-aware choices of block size, damping, MCAI construction accuracy,
and nesting are important for robust convergence.
\end{abstract}


\section{Introduction}
\label{sec:intro}

Large sparse linear systems are central to scientific computing, arising in the
numerical solution of partial differential equations (PDEs), optimization, inverse
problems, uncertainty quantification, and multiphysics simulations
\cite{benzi2005numerical,krylovinverse,nagyimagedeblurring,kalantzisUQ,keyesmultiphysics,saad20Creview,10.1007/s10915-022-01836-5,nonlinearopt}. Krylov subspace methods
are standard solvers for such systems because they exploit sparsity and avoid factorizations. 
Their practical performance, however, depends critically
on preconditioning. A useful preconditioner must improve convergence while
remaining inexpensive to apply and compatible with the target computing
architecture \cite{saaditerativebook}. Designing preconditioners that satisfy
these competing requirements remains a fundamental challenge in large-scale
scientific computing.

The growing cost of data movement further complicates preconditioner design.
On conventional von Neumann architectures, arithmetic throughput has grown
faster than the ability to move data between memory and processing units, so
memory traffic can dominate sparse and block-structured linear algebra \cite{roofline,memorywall}. Analog
in-memory computing (AIMC) and related compute-in-memory paradigms address this
bottleneck by storing matrix entries directly in memory devices and performing
matrix--vector multiplications (MVMs) through the physical dynamics of the
hardware \cite{analogfrontier,Sebastian2020}. In resistive crossbar arrays, matrix entries are encoded as
conductances, input vector components are applied as voltages or pulses, and
output currents accumulate dot products through Ohm's and Kirchhoff's laws \cite{memristorMAC}.
This makes AIMC especially relevant for algorithms in which repeated dense or
block-dense MVMs are a dominant computational component
\cite{ambrogio2023analogai,legallo2023aimcchip,wan2022rramcim}.
\edit{An introduction to AIMC devices, circuits, and architectures can be found in \cite{Ielmini2018,Sebastian2020,sensing}; a systems-level perspective can be found in \cite{deeplearninganalog}; and the non-idealities modeled in Section~\ref{sec:analog_noise_model} are discussed in detail in \cite{rasch2023hardwareaware}.}

Recent chip-scale AIMC demonstrations, driven largely by neural network
inference, have shown that analog crossbar systems can perform many MVMs with
high parallelism and potentially favorable energy behavior
\cite{ambrogio2023analogai,legallo2023aimcchip}. 
\edit{For linear solvers, architecture-level studies estimate
roughly tenfold reductions in execution time and energy for
sparse iterative solvers relative to a GPU-only system
\cite{feinberg2018}, and reductions of two to three orders
of magnitude in the time and energy of each analog subdomain
solve used for preconditioning \cite{bomananalog}.
These potential reductions in time-to-solution and energy
footprint motivate the present work;}
however, they come
with an arithmetic model that differs from conventional floating-point
computation. Analog MVMs are affected by programming variability, temporal
drift, read noise, clipping, parasitic effects, and limited analog-to-digital conversion (ADC) and digital-to-analog conversion (DAC) precision
\cite{rasch2023hardwareaware}. The computed product is therefore not an exact
MVM, but an approximate, noisy, and
limited-precision operation. This viewpoint connects AIMC to a broader theme in 
scientific computing:  the
design of iterative methods that exploit low-precision, mixed-precision, or
otherwise inexact arithmetic while preserving solver reliability
\cite{carson2018new,graillat2024adaptive,higham2021survey,mixedsvd}. 

\edit{Combining lower-precision computations with a
higher-precision outer iteration is well established in
numerical linear algebra
\cite{mixedprecisionsurvey,higham2021survey}.
A closely related digital approach is adaptive-precision
block Jacobi, which stores preconditioner blocks in different
floating-point formats to reduce data movement
\cite{adaptiveblockjacobi}.
Our setting follows a similar division of computational
roles, but analog execution introduces persistent programming
errors and application-dependent perturbations in addition
to quantization, rather than only conventional floating-point
rounding.
There is also a useful analogy with the adoption of GPUs for
scientific computing. Early GPU implementations demonstrated
their potential for sparse iterative solvers \cite{sparseGPU,li2013gpu},
while mixed-precision iterative refinement combined
lower-precision computations on the accelerator with
higher-precision residuals and solution updates
\cite{mixedprecIR}.
More recently, tensor cores developed for machine learning
have been used to accelerate mixed-precision linear solvers
\cite{tensorcore}.
Our hybrid analog--digital approach follows the same
principle: use the accelerator for suitable lower-accuracy
operations while retaining precision-sensitive computations
on the digital host.}
Given this tradeoff, the central algorithmic issue is not whether analog MVMs
can reproduce the digital counterparts exactly, but which parts of a solver can
benefit from analog execution without compromising reliability.

\edit{Early work on analog in-memory computing for
scientific applications includes the 2018 studies by
Feinberg et al.\ \cite{feinberg2018} and
Le Gallo et al.\ \cite{legallo2018}.
Feinberg et al.\ investigated high-precision sparse
iterative solvers based on fixed-point memristive MVMs,
while Le Gallo et al.\ developed a mixed-precision
approach in which an inexact in-memory solve is
refined by a high-precision digital iteration.
Subsequent studies have explored one-step matrix
inversion and matrix equation solving with feedback
circuits \cite{sunonestep,zuo2025matrixeq},
block-based analog matrix computing
\cite{pan2024blockamc}, and analog preconditioning
of iterative solvers \cite{bomananalog}.
In the present work, we use analog hardware for
block approximate inverse MVMs, while retaining
sparse matrix operations, reductions, and
precision-sensitive Krylov computations on the
digital host.}
This separation motivates a hybrid analog--digital solver architecture. 
\edit{Throughout, a hybrid analog--digital system refers to a conventional digital processor paired with an AIMC unit that executes selected MVMs.}
Within
the analog part of such an architecture, block-based decompositions provide a
practical way to map large numerical problems onto finite crossbar arrays and
to improve robustness to device and circuit non-idealities
\cite{pan2024blockamc}.

Preconditioned Krylov methods provide a concrete instance of this selective
use of analog computation. Each iteration naturally separates applications of
the coefficient matrix from applications of the preconditioner, together with
vector updates, inner products, residual computations, and orthogonalization.
The sparse MVM with the original matrix and the reduction-based operations
are irregular and precision-sensitive, and are therefore retained on the
digital host. The preconditioner action, however, can be offloaded to analog
hardware when it can be expressed primarily through MVMs. This requirement
excludes many classical choices such as incomplete-factorization preconditioners, which rely on triangular solves with sequential data dependencies and are
not naturally matched to crossbar arrays \cite{gmslr}. By contrast, approximate inverse
preconditioners can be applied
through block MVMs, making them natural candidates for analog acceleration.


In this paper we study block approximate inverse
preconditioning for flexible GMRES (FGMRES) \cite{saad1993flexible} in a hybrid analog--digital setting. We use this setting to
identify how block size, inverse construction, damping, and nesting should be
chosen when exact-arithmetic preconditioning strength must be balanced against
analog robustness.

\edit{We use exact block inverses and sparse approximate
inverses (SPAI) as reference constructions. Exact block
inverses isolate errors introduced by analog application,
while SPAI provides a standard approximate-inverse baseline.
We then study Monte Carlo approximate inverses (MCAI), with
particular emphasis on matching their construction accuracy
to the analog programming-error level. The walk length and
number of walks control the Neumann-series truncation and
sampling errors, respectively, providing explicit parameters
for this noise-matching strategy.}

The main contributions of this paper are as follows.
\begin{enumerate}
    \item We develop an analog-aware formulation of block Jacobi approximate
    inverse preconditioning. Starting from a hybrid analog--digital MVM model,
    we show that an analog preconditioner application should be viewed as an
    input- and application-dependent perturbed block action, rather
    than as the application of a deterministic matrix. This leads to a different
    preconditioner design criterion which takes into consideration that the 
    block size, density, inverse accuracy, and
    scaling must be selected by balancing exact-arithmetic preconditioning
    strength against analog write noise, input/output perturbations,
    quantization, clipping, and dynamic range constraints.

    \item We adapt MCAI preconditioners based on
    Neumann series \cite{regenerative} to the analog setting. Exact block inverses and
    sparse approximate inverses are used as reference constructions, while MCAI is
    emphasized as an analog-suitable construction because its accuracy is tunable
    and its setup consists of highly parallelizable random walks. We provide a
    noise-matching error bound that explains how the number of walks and the walk
    length should be chosen so that inverse-construction error is commensurate with
    the analog realization error rather than machine precision.
 \item We introduce robustness mechanisms for analog block preconditioning,
    including damping, shifting, and nested block solves. Damping moderates
    unreliable analog inverse actions, while nesting reduces the size of the
    matrices programmed onto crossbar arrays and thereby reduces analog
    exposure. These mechanisms are designed specifically around the
    perturbation structure of analog MVMs.

    \item We present numerical experiments with realistic analog MVM
    simulations on Poisson and convection--diffusion--reaction problems. The
    results show non-monotonic block-size behavior, robustness gains from modest
    damping, saturation of MCAI construction accuracy once the analog noise floor is reached, and nested analog preconditioning that closely matches
    the iteration counts of corresponding digital block Jacobi
    preconditioners on larger Poisson problems.
\end{enumerate}

The remaining sections are organized as follows.
Section~\ref{sec:analog_model} introduces the hybrid analog--digital MVM model,
including the crossbar abstraction, hybrid solver architecture, and analog
noise and quantization model. Section~\ref{sec:alg} develops the proposed
analog-aware block approximate inverse preconditioners. Section~\ref{sec:experiments} presents
numerical experiments that evaluate the resulting preconditioners on model PDE
problems. Section~\ref{sec:conc} concludes with limitations and directions for
future work.

\section{Hybrid Analog--Digital MVM Model}
\label{sec:analog_model}

This section defines the analog MVM abstraction
used throughout the paper. The model is not intended to capture every
device- and circuit-level effect in a resistive crossbar array. Instead, it
captures the features most relevant to Krylov preconditioning, i.e., analog MVMs are
highly parallel but noisy. 

\subsection{Noiseless crossbar MVM}
\label{sec:crossbar_mvm}

A resistive crossbar array is a 2D grid of programmable memory
devices located at the intersections of horizontal word lines and vertical bit
lines. Each device stores a programmable conductance. In the MVM mode, the entries
of an input vector are encoded as voltages or pulses and applied along the word
lines. The resulting currents are accumulated along the bit lines according to
Kirchhoff's current law and are then sensed by peripheral circuitry \cite{sensing}.

We use the following convention throughout the paper. The input component
\(x_j\) is applied to word line \(j\), and the output component
\(\widetilde y_i\) is collected from bit line \(i\) after sensing and
digitization. 
Let 
$G^{\rm phys} \in \mathbb{R}^{n_w \times n_b},$ 
where
$n_w$ denotes the number of word lines,
$n_b$ the number of bit lines, and
\(G^{\rm phys}_{ji} \in \mathbb{R}\)  the conductance at word line
\(j\) and bit line \(i\). \edit{This indexing follows the physical layout of the crossbar, with
rows corresponding to word lines and columns to bit lines; hence
the input--output relation involves \((G^{\rm phys})^T\).}

When applying the voltage vector \(v\) to the word
lines, the ideal bit-line current is
\begin{equation}
    f_i = \sum_j G^{\rm phys}_{ji} v_j ,
\end{equation}
or equivalently, in the MVM form,
\begin{equation}
\label{eq:f=GTv}
    f=(G^{\rm phys})^T v, \quad 
    f\in \mathbb{R}^{n_b}.
\end{equation}

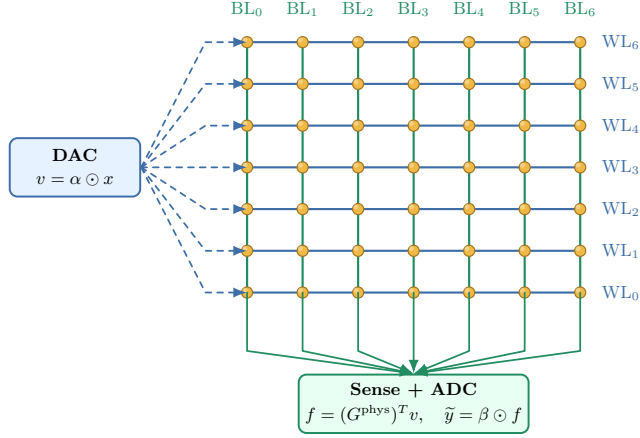
\begin{figure}[t]
\centering
\resizebox{0.65\linewidth}{!}{%
\begin{tikzpicture}[
  font=\small,
  >=Latex,
  line cap=round,
  line join=round
]
\definecolor{cBorder}{RGB}{40,55,75}
\definecolor{cWL}{RGB}{62,110,165}
\definecolor{cBL}{RGB}{32,140,95}
\definecolor{cCell}{RGB}{240,185,70}
\definecolor{cCellEdge}{RGB}{160,115,30}
\definecolor{cBoxA}{RGB}{230,242,255}
\definecolor{cBoxB}{RGB}{232,255,244}

\def\nrows{7}
\def\ncols{7}
\def\dx{1.00}
\def\dy{0.75}

\foreach \r in {0,...,\numexpr\nrows-1\relax} {
  \draw[line width=1.1pt, draw=cWL]
    (0,{\r*\dy}) -- ({(\ncols-1)*\dx},{\r*\dy});
  \node[anchor=west, text=cWL, fill=white, inner sep=1pt]
    at ({(\ncols-1)*\dx + 0.35},{\r*\dy}) {$\mathrm{WL}_{\r}$};
}

\foreach \c in {0,...,\numexpr\ncols-1\relax} {
  \draw[line width=1.1pt, draw=cBL]
    ({\c*\dx},0) -- ({\c*\dx},{(\nrows-1)*\dy});
  \node[anchor=south, text=cBL]
    at ({\c*\dx},{(\nrows-1)*\dy+0.35}) {$\mathrm{BL}_{\c}$};
}

\foreach \r in {0,...,\numexpr\nrows-1\relax} {
  \foreach \c in {0,...,\numexpr\ncols-1\relax} {
    \filldraw[fill=cCell, draw=cCellEdge, line width=0.6pt]
      ({\c*\dx},{\r*\dy}) circle (0.10);
    \fill[white, opacity=0.35]
      ({\c*\dx-0.04},{\r*\dy+0.04}) circle (0.045);
  }
}

\node[
  draw=cWL, fill=cBoxA, line width=1.0pt,
  rounded corners=4pt,
  minimum width=2.35cm, minimum height=1.05cm,
  align=center
] (dac) at (-3.1,{0.5*(\nrows-1)*\dy})
{\textbf{DAC}\\[1pt]\small $v=\alpha \odot x$};

\foreach \r in {0,...,\numexpr\nrows-1\relax} {
  \draw[->, draw=cWL, line width=0.9pt, dashed]
    (dac.east) -- (-0.75,{\r*\dy}) -- (0,{\r*\dy});
}

\node[
  draw=cBL, fill=cBoxB, line width=1.0pt,
  rounded corners=4pt,
  minimum width=3.4cm, minimum height=1.05cm,
  align=center
] (adc) at ({0.5*(\ncols-1)*\dx},-2.0)
{\textbf{Sense + ADC}\\[2pt]\small
$f=(G^{\rm phys})^T v,\quad \widetilde y=\beta \odot f$};

\foreach \c in {0,...,\numexpr\ncols-1\relax} {
  \draw[->, draw=cBL, line width=0.9pt]
    ({\c*\dx},0) -- ({\c*\dx},-1.05) -- (adc.north);
}

\end{tikzpicture}
}
\caption{A \(7\times 7\) analog crossbar MVM. The DAC maps the digital input
to word-line signals \(v=\alpha\odot x\); the crossbar accumulates bit-line
currents \(f=(G^{\rm phys})^T v\); and the sense/ADC circuitry returns
\(\widetilde y=\beta\odot f\).}
\label{fig:crossbar-7x7}
\end{figure}

Figure~\ref{fig:crossbar-7x7} illustrates this signal path. The digital input
vector \(x\) is converted by DAC circuitry into an analog word-line signal
\begin{equation}
    v = \alpha \odot x ,
\end{equation}
where \(\alpha\) is a vector of input scaling factors and \(\odot\) denotes
componentwise multiplication. The crossbar produces
$f$ as in \eqref{eq:f=GTv},
and the sensing and ADC circuitry returns the digital output
\begin{equation}
    \widetilde y = \beta \odot f ,
\end{equation}
where \(\beta\) is a vector of output scaling factors. 
Thus, the physical cell at word line
\(j\) and bit line \(i\) stores the contribution of input component \(x_j\) to
output component \(y_i\),
up to input and output scaling.
Combining these stages, it follows that
\begin{equation}
    \label{eq:idealmap}
    \widetilde y
    =
    \beta \odot \left((G^{\rm phys})^T(\alpha\odot x)\right)
    =
    D_\beta (G^{\rm phys})^T D_\alpha x ,
\end{equation}
where \(D_\alpha=\operatorname{diag}(\alpha)\) and
\(D_\beta=\operatorname{diag}(\beta)\)
denote the diagonal matrices with $\alpha$ and $\beta$ on their diagonals, respectively.

Because physical conductances are nonnegative, signed matrix entries are
typically represented using differential pairs. One logical weight is then
represented by two physical conductances, 
namely,
\(G^{+,\rm phys}_{ji}\) and
\(G^{-,\rm phys}_{ji}\), with effective signed conductance proportional to
\begin{equation}
G^{\rm phys}_{ji} \propto~
G^{+,\rm phys}_{ji} - G^{-,\rm phys}_{ji}.
\end{equation}
Thus, each drawn cell in Figure~\ref{fig:crossbar-7x7} should be interpreted as
one logical matrix entry; in a signed implementation, that entry may be realized
by a differential pair.
For a fixed crossbar tile and fixed peripheral precision, array-level MVM
latency is
\emph{largely independent of the logical density of the stored matrix
block}.
This makes crossbars attractive for dense or block-dense approximate
inverse preconditioners whose application can be expressed primarily as
MVMs. 

\subsection{Hybrid preconditioning architecture}
\label{sec:hybrid_architecture}

\begin{figure}[t]
    \centering
    \includegraphics[width=0.92\linewidth]{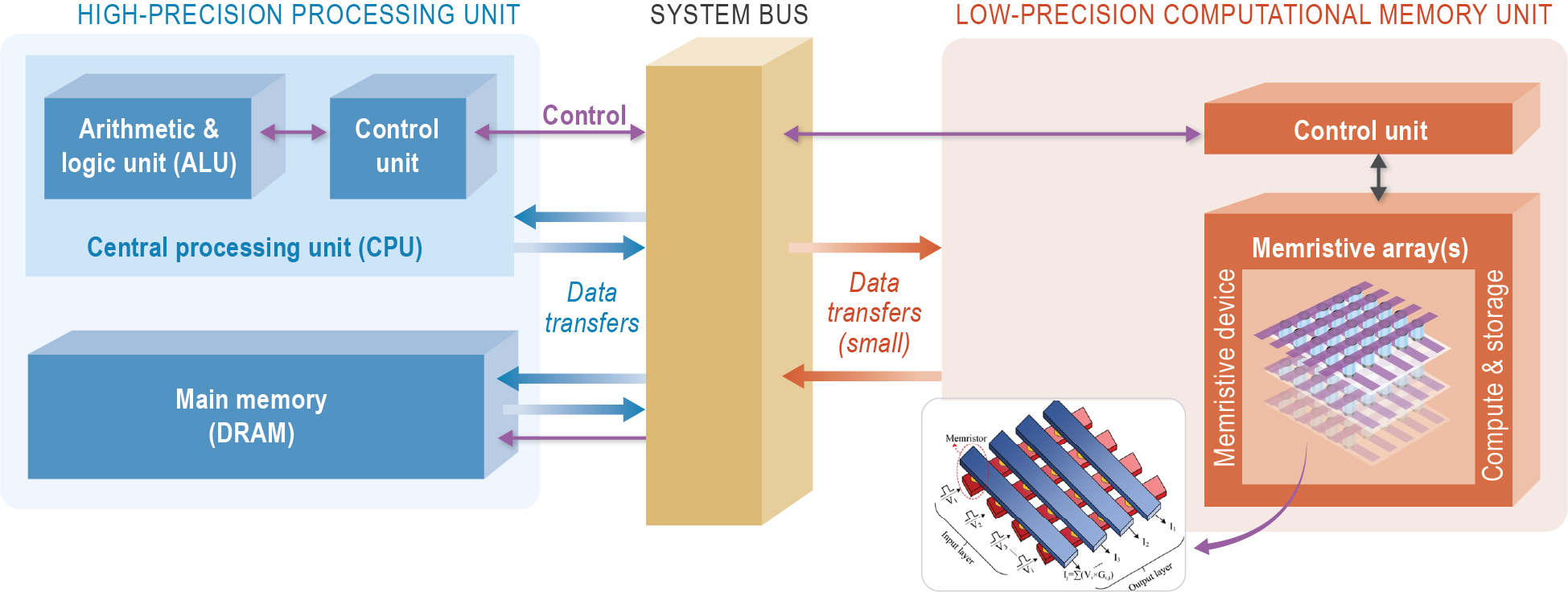}
    \caption{Hybrid analog--digital solver architecture. The digital host
    manages the Krylov iteration and sparse operations with \(A\), while the
    computational memory unit stores block approximate inverse preconditioners
    and applies them through approximate analog MVMs.}
    \label{fig:hyb}
\end{figure}
The hybrid architecture used in this work is shown schematically in
Figure~\ref{fig:hyb}. The high-precision digital host stores the sparse
coefficient matrix \(A\), maintains the Krylov basis, and performs the
operations for which sparsity, reductions, and high precision are important. 
These operations include 
sparse MVMs with \(A\), residual computations, inner
products, vector updates, orthogonalization, and 
solving the small least-squares
problem in FGMRES. The low-precision computational memory unit
contains memristive crossbar arrays and is used for 
preconditioner applications dominated by MVM.

Once a preconditioner block has been programmed into the array,
repeated applications require 
only
vector transfers through the system bus, while
the preconditioner matrix remains stored in the computational memory unit.
\edit{Programming a block is considerably more expensive than applying it \cite{analogwritetime}, so its cost must be amortized over the tens to hundreds of preconditioner applications in a single linear system solve, and further amortized across multiple right-hand side vectors or sequences of related linear systems.}
The specific block Jacobi preconditioners used in this architecture are introduced
in Section~\ref{sec:alg}.

\subsection{Analog MVM noise and quantization model}
\label{sec:analog_noise_model}

Section~\ref{sec:crossbar_mvm} described the ideal crossbar MVM
\[
    x \ \mapsto \ v=\alpha\odot x
    \ \mapsto \ f=(G^{\rm phys})^T v
    \ \mapsto \ \widetilde y=\beta\odot f,
\]
which, after scaling and programming, realizes the logical operation \(y=Mx\).
Actual analog MVMs are not exact realizations of this ideal map. Device
variability, finite conductance resolution, input/output conversion, limited
dynamic range, clipping, and readout noise perturb the computed product. These
non-idealities are central to analog preconditioner design, and can sometimes be 
counterintuitive. For example, a block inverse
that is strong in digital
arithmetic may be less effective once the programmed
matrix and each MVM application are perturbed.

In this subsection, \(M\) denotes the logical matrix to be applied by one
analog MVM. Before programming, \(M\) is scaled by a factor \(\mu>0\) so that the entries of
\(\mu^{-1}M\) fall within 
the conductance range representable by the crossbar. The ADC
and digital post-processing then return the result in the original logical
scaling. With this convention, the non-ideal analog MVM is modeled as
\begin{equation}
\label{eq:analog_mvm_full}
    \widetilde y
    =
    \mu \operatorname{ADC}
    \left(
    \left[
        \left(\mu^{-1}M\right)\odot
        \left(\mathbf{1}+N^{(M)}\right)
        +N^{(W)}
    \right]
    \left(
        \operatorname{DAC}(x)+n^{(I)}
    \right)
    +n^{(O)}
    \right),
\end{equation}
where \(N^{(M)}\) represents multiplicative conductance-dependent perturbations,
\(N^{(W)}\) represents additive write perturbations in the scaled programmed
matrix, and \(\mathbf{1}\) is the all-ones matrix. The vectors \(n^{(I)}\) and
\(n^{(O)}\) represent input-referred and output-referred noise, respectively.
The \(\operatorname{DAC}\) model includes input scaling and quantization, while
the \(\operatorname{ADC}\) model includes output quantization, finite dynamic
range, and clipping. The initial matrix scaling is undone after the ADC, rescaling the output vector to logical coordinates. 
Thus,
\eqref{eq:analog_mvm_full} is the noisy counterpart of the ideal signal path in
Figure~\ref{fig:crossbar-7x7}, with the additional fact that the logical
matrix \(M\) is programmed as \(\mu^{-1}M\) and the output vector is rescaled after readout.
Figure~\ref{fig:noisemodel} summarizes these perturbation sources and their
locations in the analog MVM pipeline.
\begin{figure}[t]
    \centering
    \includegraphics[width=0.95\linewidth]{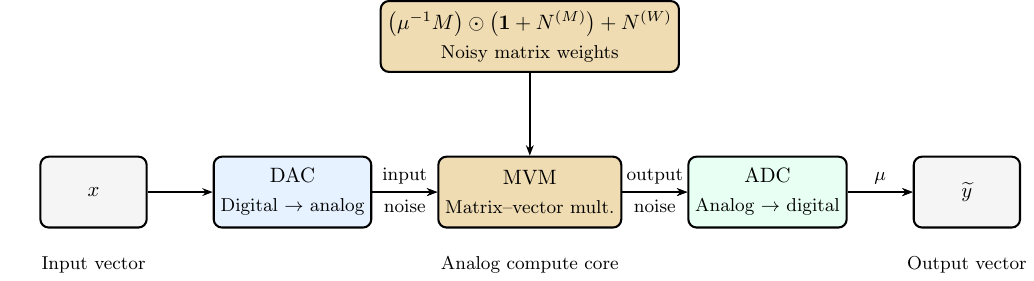}
\caption{Noise and quantization sources in a single analog MVM,
\(\widetilde y\approx Mx\). The logical matrix \(M\) is first scaled and
programmed as \(\mu^{-1} M\). The scaled programmed matrix is affected by
multiplicative perturbations \(N^{(M)}\) and additive write perturbations
\(N^{(W)}\); the input and output paths are affected by DAC/ADC quantization,
finite dynamic range, and input/output noise. The output is rescaled by
\(\mu\) in logical coordinates.}
    \label{fig:noisemodel}
\end{figure}

Equivalently, in programmed coordinates, a matrix entry is represented as
\begin{equation}
    \widetilde M^{\rm prog}_{ij}
    =
    {\mu^{-1}}M_{ij}(1+\epsilon_{ij})+\zeta_{ij},
\end{equation}
where \(\epsilon_{ij}\) is a relative perturbation and \(\zeta_{ij}\) is an
absolute write perturbation. After rescaling by \(\mu\), the corresponding
logical-coordinate entry is
\begin{equation}
\label{eq:logcoord}
    \widetilde M^{\rm log}_{ij}
    =
    M_{ij}(1+\epsilon_{ij})+\mu\zeta_{ij}.
\end{equation}
Thus, additive write noise in the programmed conductance domain is multiplied by
the matrix scaling factor when viewed in logical coordinates. The additive write noise is applied to all matrix entries, even zero entries unless they are explicitly masked; moreover, small entries may be overwhelmed by this absolute perturbation. In contrast, multiplicative noise is applied proportionally to each entry, so small entries receive proportionally small perturbations and zero entries are left unchanged, preserving the sparsity pattern of the matrix. 

\edit{Our model builds on the analog preconditioning framework
in \cite{kalantzis2021analog,kalantzis2023fgmres}.
For the present study, we set \(N^{(M)}=0\) and focus on
additive write noise, input/output noise, quantization, and
clipping. This choice allows us to examine how persistent
programming errors and application-dependent perturbations
interact with block size, scaling, and inverse construction;
it does not assume that multiplicative errors are absent
from physical devices.
We also omit conductance drift, wire resistance, and other
array-level parasitic effects. Global drift compensation
has been demonstrated for phase-change memory devices
\cite{Joshi2020}, but it does not eliminate device-to-device
drift variability. These omissions limit the scope of the
numerical conclusions: the experiments assess robustness
under the specified perturbation model, rather than under
all non-idealities of a particular hardware implementation.}

The DAC and ADC are modeled as finite-bit quantizers with finite dynamic
range. Before DAC quantization, the input vector is scaled so that its entries
lie in \([-1,1]\), and each scaled entry is mapped to the nearest of
\(2^{\ell_{\rm DAC}}-1\) uniformly spaced levels. The analog output is
quantized by the ADC over its prescribed dynamic range, after which the input vector and matrix scalings are undone. \edit{The default simulator parameters follow previous analog preconditioning studies~\cite{kalantzis2021analog,kalantzis2023fgmres} and are listed in Table~\ref{tab:noisevalues}. They are not intended to represent a calibrated model of a particular hardware implementation.}

\begin{table}[t]
\centering
\caption{\edit{Default} analog simulator parameters. The matrix scaling factor
\(\mu\) is chosen per programmed block so that \(\mu^{-1}M\) lies within the
representable conductance range \edit{$[-0.6, 0.6]$}.}\label{tab:noisevalues}
\begin{tabular}{c|c|c}
\textbf{Parameter} & \textbf{Symbol} & \textbf{Value} \\ \hline
Additive write noise standard deviation & \(\sigma_W\) & \(5\cdot 10^{-3}\) \\ 
Input noise standard deviation & \(\sigma_I\) & \(1\cdot 10^{-2}\) \\ 
Output noise standard deviation & \(\sigma_O\) & \(1\cdot 10^{-2}\) \\ 
Matrix scaling factor & \(\mu\) & \(\max_{i,j}|M_{ij}|/0.6\) \\ 
Number of DAC bits & \(\ell_{\rm DAC}\) & \(7\) \\ 
Number of ADC bits & \(\ell_{\rm ADC}\) & \(9\) \\ \hline
\end{tabular}
\end{table}

For analysis, we collect DAC quantization and input noise into an effective
input perturbation \(e_I(x)\), and we absorb output noise, ADC quantization,
and clipping into an output perturbation
\(e_O(x)\). With \(N^{(M)}=0\), the analog MVM can be written as
\begin{equation}
\label{eq:analog_mvm_simplified}
    \widetilde y
    =
    \left(M+\mu N^{(W)}\right)\left(x+e_I(x)\right)+\mu e_O(x).
\end{equation}
Expanding yields
\begin{equation}
\label{eq:analog_mvm_error_decomposition}
    \widetilde y
    =
    Mx
    + \mu N^{(W)}x
    + M e_I(x)
    + \mu N^{(W)}e_I(x)
    + \mu e_O(x).
\end{equation}
Thus, an analog MVM can be written abstractly as
\begin{equation}
\label{eq:analog_operator_abstraction}
    \widetilde y = Mx+\delta_M(x),
\end{equation}
where \(\delta_M(x)\) depends on the programmed matrix, the input vector, the
quantization and clipping rules, and the random noise sampled during the MVM.
The perturbation is therefore generally input-dependent and
application-dependent. Since \(e_I(x)\) and \(e_O(x)\) include quantization and
possible clipping, they need not be Gaussian or unbiased even if the underlying
input and output noise terms are modeled as mean-zero random variables.

\edit{The terms in \eqref{eq:analog_mvm_error_decomposition} separate three sources of error.
The write perturbation \(\mu N^{(W)}\) is sampled once
when a block is programmed and remains fixed throughout
a solve. In contrast, the input and output noise terms
\(n^{(I)}\) and \(n^{(O)}\) are resampled at each MVM
application. Quantization and clipping introduce additional
input-dependent errors: within the representable range,
quantization resembles fixed-point rounding, whereas
clipping limits values outside that range.
Thus, even with a fixed programmed matrix, the realized
preconditioner action can vary between applications.
We use FGMRES to accommodate this variation while retaining
the products with the coefficient matrix and the
precision-sensitive operations on the digital host.
Finally, programming \(\mu^{-1} M\) and rescaling the output
convert the additive programming perturbation \(N^{(W)}\)
into \(\mu N^{(W)}\) in logical coordinates, as shown in
\eqref{eq:logcoord}. Consequently, the absolute write-error scale depends
on the scaling required to represent the preconditioner.}

\section{Analog-Aware Block Approximate Inverse Preconditioners}
\label{sec:alg}

The analog MVM model in Section~\ref{sec:analog_model} changes the design
criteria for block Jacobi preconditioning. In exact arithmetic, a block Jacobi
preconditioner is usually judged by how well the diagonal block inverses
approximate the action of the inverse of the coefficient matrix. In the analog
setting, however, each block action is realized through a noisy, quantized, and
input-dependent MVM. 
This section compares several block inverse constructions from this
analog-aware viewpoint. \edit{Exact block inverses and SPAI serve as reference
constructions. Our focus is on noise-matched MCAI construction,
in which the walk length and number of walks provide explicit
controls for balancing inverse-construction accuracy against
analog programming error.}
We
also discuss damping and nested block solves as 
mechanisms for improving robustness. 
\subsection{Inexact analog block Jacobi}
\label{sec:block_jacobi}

Let \(A=[A_{ij}]_{i,j=1}^p\) be a \(p\times p\) block partition of
\(A\in\mathbb{R}^{n\times n}\). The block Jacobi approximate inverse
preconditioner has the form
\begin{equation}
\label{eq:block_jacobi_M}
    M=\operatorname{blkdiag}(M_{11},\ldots,M_{pp}),
    \quad \mbox{with} \quad
    M_{kk}\approx A_{kk}^{-1}.
\end{equation}
For a conformally partitioned vector \(r^T=(r_1^T,\ldots,r_p^T)\), the ideal action of $M$ is
\begin{equation}
\label{eq:block_jacobi_action}
    Mr =
    \begin{bmatrix}
        M_{11}r_1\\
        \vdots\\
        M_{pp}r_p
    \end{bmatrix}.
\end{equation}
In the analog implementation, each block product is performed by an analog
MVM. Using the perturbation model of Section~\ref{sec:analog_noise_model}, the
realized block action is
\begin{equation}
\label{eq:analog_block_action}
    \widetilde z_k
    =
    M_{kk}r_k+\delta_{M_{kk}}(r_k),
\end{equation}
and hence
\begin{equation}
\label{eq:analog_block_jacobi_action}
    \widetilde z
    =
    Mr+\delta_M(r).
\end{equation}
The perturbation \(\delta_M(r)\) is blockwise, input-dependent, and may vary
from one preconditioner application to the next.

For block $k$, let \(\mu_k\) be the scaling factor used to program
\(\mu_k^{-1}M_{kk}\). 
From \eqref{eq:analog_mvm_simplified}, we can write
the simplified analog block action as
\begin{align}
\label{eq:block_error_expansion}
    \widetilde z_k
    =
    \left(M_{kk}+\mu_k N_k^{(W)}\right)
    \left(r_k+e_I(r_k)\right)
    +\mu_k e_O(r_k) 
    =M_{kk}r_k + \delta_{M_{kk}}(r_k),
\end{align}
in which
\begin{equation}
\label{eq:deltaMkk}
    \delta_{M_{kk}}(r_k)
    =
    \mu_k N_k^{(W)}r_k
    +M_{kk}e_I(r_k)
    +\mu_k N_k^{(W)}e_I(r_k)
    +\mu_k e_O(r_k).
\end{equation}
Thus, improving \(M_{kk}\approx A_{kk}^{-1}\) in exact arithmetic is only one
part of the analog preconditioning design. 
The realized action also depends on how the chosen
block inverse amplifies write noise, input perturbations, output
perturbations, quantization, clipping, and dynamic range effects.

The additive write noise is applied to the scaled programmed block
\(\mu_k^{-1}M_{kk}\), but appears as \(\mu_k N_k^{(W)}\) in logical coordinates
after rescaling. This scaling affects the absolute perturbation introduced by
the analog write process. If
\(N_k^{(W)}\in\mathbb{R}^{n_k\times n_k}\) has independent mean-zero entries
with common variance \(\sigma_W^2\), and \(r_k\) is fixed with respect to the
write noise, then
\begin{equation}
\label{eq:write_noise_block_scaling}
    \mathbb{E}\!\left[
    \|\mu_k N_k^{(W)}r_k\|_2^2 \,\middle|\, r_k
    \right]
    =
    n_k\mu_k^2\sigma_W^2\|r_k\|_2^2 ,
\end{equation}
where \(n_k\) is the block size. 

Equivalently, the input-normalized 
root-mean-square (RMS) write-noise contribution is
\[
    \frac{
    \sqrt{\mathbb{E}\!\left[
    \|\mu_k N_k^{(W)}r_k\|_2^2 \,\middle|\, r_k
    \right]}
    }{\|r_k\|_2}
    =
    \sqrt{n_k}\,\mu_k\sigma_W .
\]
Thus, per unit input norm, the absolute write-noise contribution grows like
\(\sqrt{n_k}\mu_k\). The RMS write perturbation relative to the
ideal preconditioned vector is
\[
    \eta_k^{(W)}(r_k)
    :=
    \frac{
    \sqrt{\mathbb{E}\!\left[
    \|\mu_k N_k^{(W)}r_k\|_2^2 \,\middle|\, r_k
    \right]}
    }{\|M_{kk}r_k\|_2}
    \geq 
    \frac{\sqrt{n_k}\,\mu_k\sigma_W}{\|M_{kk}\|_2} .
\]
\edit{For fixed \(\sigma_W\), this lower bound on the
relative write perturbation increases with block size
whenever \(\sqrt{n_k}\mu_k/\|M_{kk}\|_2\) increases.}
More generally, the relative effect depends on the
interaction between the scaling factor \(\mu_k\), the block size \(n_k\), and
the action of \(M_{kk}\). This provides a
mechanism by which larger blocks can become more sensitive to analog write
perturbations, even though they may improve the exact-arithmetic block Jacobi
approximation.

\subsubsection{A note on low-precision preconditioning}
Analog construction is more constrained than rounding a
preconditioner to low precision. In a standard low-precision digital setting,
one may construct \(M_{kk}\) and store or apply a rounded matrix
\(Q_\ell(M_{kk})\), while preserving the digital sparsity pattern and retaining
standard operations such as triangular solves. In the analog setting, the block
must be mapped to feasible conductance values with input/output scalings;
additive write noise can dominate small entries; logical zeros need not be
preserved unless they are explicitly masked or left unprogrammed; and the
preconditioner must be applied by MVMs rather than triangular solves.
Consequently, analog-aware preconditioner construction must balance
exact-arithmetic quality with block size, norm, density, scaling, tile
constraints, and expected read/write perturbations.

\subsection{Reference block inverse constructions}
\label{sec:inverse_construction}

For each diagonal block \(A_{kk}\), we construct an approximate inverse
\(M_{kk}\). We first describe two reference choices: (a) exact block inverses and (b) sparse approximate inverses. These baselines separate
exact-arithmetic preconditioning strength from analog realization error and
provide points of comparison for the MCAI construction in
Section~\ref{sec:mcmc}.

The first reference choice is the exact block inverse,
\[
    M_{kk}=A_{kk}^{-1},
\]
computed in double precision before being programmed onto analog hardware.
This choice gives a strong block Jacobi preconditioner in exact arithmetic and
is therefore useful as a benchmark. However, exact block inverses are generally
dense and may be more accurate than the analog hardware can reliably realize.
Thus, exact block inverses are not
necessarily optimal analog preconditioners, even when they are strong in exact digital
arithmetic.

The second reference choice is sparse approximate inverse (SPAI) preconditioning 
\cite{benzi1999comparative,gouldscott,grote1997parallel}. For a generic diagonal
block \(C=A_{kk}\), let \(M_C\approx C^{-1}\) denote a sparse approximate
inverse with prescribed sparsity pattern \({\cal S}\). It is obtained by
solving
\begin{equation}
\label{eq:spai_main}
    \min_{\operatorname{supp}(M_C)\subseteq{\mathcal S}}
    \|I-CM_C\|_F^2, \quad \operatorname{supp}(M_C) = \{(i,j)\mid [M_C]_{ij}\neq 0\}.
\end{equation}
This problem decomposes by columns into independent least-squares problems and
is attractive for parallel construction. Sparse approximate inverses are also
tunable since the sparsity pattern, maximum number of nonzeros, dropping strategy,
and residual tolerance can be used to trade exact-arithmetic quality against
density and setup cost.

In digital arithmetic, the main advantage of SPAI is that sparsity reduces
both storage and MVM cost. In the analog setting, its role is more subtle. For
a fixed crossbar tile, the array-level MVM latency is largely independent of
the logical density of the stored block, while perturbations and peripheral
effects depend on the physical realization. Sparsity may still be beneficial,
for example, by reducing current draw and the number
of entries exposed to write noise when zeros are explicitly masked. However,
digital sparsity alone is not a complete analog design criterion.

The analog perturbation model changes how approximate inverse preconditioners should be
judged. Suppose \(M_C\) is an approximate inverse for a block \(C\), but the
programmed analog block is realized in logical coordinates as
\[
    \widetilde M_C=M_C+E,
\]
so that
\[
    C\widetilde M_C = CM_C+CE .
\]
Thus, the realized preconditioner depends both on the exact-arithmetic quality
of \(M_C\) and on the hardware perturbation \(E\). Therefore, reducing the 
residual \(\|I-CM_C\|_F\) far below the analog
perturbation scale may not improve the realized analog preconditioner. A
noise-aware SPAI strategy is possible, but the relation between sparsity
choices, scaling, and analog realization error is indirect. The MCAI
construction in the next subsection provides a complementary mechanism in
which truncation and sampling errors can be matched directly to the analog
noise floor.

\subsection{Noise-matched Monte Carlo approximate inverse preconditioners}
\label{sec:mcmc}

We now consider a complementary approximate inverse construction whose error
can be controlled explicitly. Monte Carlo approximate inverses based on
Neumann series~\cite{ALEXANDROV1998113,regenerative,doi:10.1137/130904867} have recently
been studied as highly parallel preconditioners for digital iterative solvers.
To our knowledge, their behavior as analog preconditioners has not previously
been analyzed. In the present setting, their advantage is not that they are
uniformly more accurate than exact block inverses or sparse approximate
inverses. Rather, their setup error has two transparent controls: the walk
length controls truncation error in the Neumann series, and the number of
walks controls Monte Carlo sampling error. These parameters can therefore be
chosen so that the construction accuracy is commensurate with the analog
realization error rather than machine precision. We refer to the resulting
block approximate inverse preconditioners as Monte Carlo approximate inverse
(MCAI) preconditioners.

Let \(C=A_{kk}\in\mathbb{R}^{n_k\times n_k}\) be a diagonal block whose inverse
we want to approximate. Choose \(\gamma>0\) and define
\[
    H=I-\gamma C .
\]
Then \(I-H=\gamma C\), and hence
\[
    C^{-1}=\gamma(I-H)^{-1}.
\]
If \(\rho(H)<1\), the Neumann series converges and
\begin{equation}
\label{eq:neumann_block_inverse}
    C^{-1}
    =
    \gamma\sum_{t=0}^{\infty}H^t .
\end{equation}
Thus, constructing an approximate inverse reduces to estimating entries of the
powers \(H^t\) and summing them.

The Monte Carlo estimator uses the path interpretation of matrix powers. For
\(t\geq 1\), the entry \([H^t]_{ij}\) is the sum over all paths of length \(t\),
$i_0 \rightarrow i_1 \rightarrow \ldots \rightarrow i_{t}$, where $i_0=i$ and $i_{t}=j$ (i.e., paths from index \(i\) to index \(j\)):
\begin{equation} 
\label{eq:path_expansion}
[H^t]_{ij}
=
\sum_{\substack{i_0,i_1,\ldots,i_t \\ i_0 = i, \ i_t = j}}
H_{i_0i_1}H_{i_1i_2}\cdots H_{i_{t-1}i_t}.
\end{equation}

Instead of enumerating all such paths, a Markov chain samples them randomly.
Let \(P\) be a transition probability matrix satisfying \(P_{pq}>0\) whenever
\(H_{pq}\neq 0\), and $\sum_q P_{pq}=1$.
In the experiments, we use
\begin{equation}
\label{eq:mcmc_transition}
    P_{pq}
    =
    \frac{|H_{pq}|}{\sum_{q'} |H_{pq'}|}
\end{equation}
for rows with a nonzero denominator. If a row of \(H\) has a zero denominator (i.e., a zero row), 
then
there are no outgoing nonzero path contributions from that state. In the
implementation, such a walk may be terminated; equivalently, all subsequent
weights are set to zero.

For a walk initialized at \(X(0)=i\), define
\begin{equation}
\label{eq:mcmc_weight}
    W_0=1,
    \qquad
    W_t
    =
    W_{t-1}
    \frac{H_{X(t-1),X(t)}}{P_{X(t-1),X(t)}} .
\end{equation}
The weight \(W_t\) corrects for the fact that paths are sampled using
transition probabilities \(P_{pq}\), whereas the desired path contributions in
\eqref{eq:path_expansion} involve the matrix entries \(H_{pq}\). This
importance-sampling correction gives
\begin{equation}
\label{eq:mcmc_power_unbiased}
    \mathbb{E}\!\left[
        W_t \mathbf{1}_{\{X(t)=j\}}
        \,\middle|\, X(0)=i
    \right]
    =
    [H^t]_{ij},
\end{equation}
which is Lemma 2.1 in \cite{regenerative}.
Herein, we assume that the walk length \(\ell\) 
is a fixed truncation parameter rather than 
a random stopping time. For each row \(i\), we generate \(n_c\) independent walks, each run for
at most \(\ell\) transitions, where the prefix of
length \(t\) estimates \([H^t]_{ij}\), and the sum over
\(t=0,\ldots,\ell\) estimates the truncated Neumann series. Moreover, we define
\begin{equation}
\label{eq:mcmc_inverse_estimator}
    [M_{\ell,n_c}^{\rm MC}]_{ij}
    =
    \gamma\,
    \frac{1}{n_c}
    \sum_{s=1}^{n_c}
    \sum_{t=0}^{\ell}
    W_t^{(s)}
    \mathbf{1}_{\{X_s(t)=j\}} .
\end{equation}
By linearity of expectation and \eqref{eq:mcmc_power_unbiased},
\begin{equation}
\mathbb{E}\!\left[
[M_{\ell,n_c}^{\rm MC}]_{ij}
\right]
=
\gamma
\sum_{t=0}^{\ell}
[H^t]_{ij}
=
[S_\ell]_{ij}, 
\qquad
S_\ell=\gamma\sum_{t=0}^{\ell}H^t ~\approx C^{-1}.
\end{equation}
Therefore, \(M_{\ell,n_c}^{\rm MC}\) is unbiased for the truncated Neumann
approximation \(S_\ell\). Its error relative to \(C^{-1}\) has two algorithmic
sources: first, the truncation error of the infinite Neumann series, controlled by \(\ell\), and second,
the Monte Carlo sampling error, controlled by \(n_c\). \edit{Constructing \(M_{\ell,n_c}^{\rm MC}\) uses \(n_c\)
walks for each of the \(n_k\) rows, with at most \(\ell\)
transitions per walk, giving at most \(n_k n_c\ell\)
simulated transitions per block. The walks can be generated
independently in parallel. For fixed \(\ell\), increasing
\(n_c\) reduces the RMS sampling error at a rate
\(n_c^{\nicefrac{-1}{2}}\), but does not reduce the truncation error.
Increasing \(\ell\) reduces the truncation-error bound,
but increases the work per walk and potentially the
sampling variance. The noise-matching criterion therefore
guides the joint choice of these parameters, but does not
eliminate the potentially high setup cost for slowly
converging Neumann series.}

\edit{The Neumann series
\eqref{eq:neumann_block_inverse} converges if and only if
\(\rho(H)<1\), whereas Theorem~\ref{thm:noise_matched_mcmc} assumes
the stronger condition \(\|H\|_2<1\); \(\rho(H)=\|H\|_2\) when \(C\) is
symmetric. For symmetric positive-definite (SPD) \(C\),
\(\rho(H)<1\) if and only if
\(0<\gamma<2/\lambda_{\max}(C)\).
Hence, for any upper bound \(L\geq\lambda_{\max}(C)\), such as
\(L=\|C\|_\infty\), choosing \(\gamma^{-1}=\tau L\) with \(\tau>1/2\)
ensures the contraction condition without computing eigenvalues.
If \(\tau>1\), all eigenvalues of \(H\) lie in \((0,1)\), and
\(\rho(H)=1-\gamma\lambda_{\min}(C)<1\).
A quantitative truncation-error bound still requires an explicit upper
bound on the contraction factor; when \(\tau>1\), this can be obtained
from a positive lower bound on \(\lambda_{\min}(C)\).
For general, possibly nonsymmetric \(C\), there exists a \(\gamma>0\)
satisfying \(\|I-\gamma C\|_2<1\) if and only if the symmetric part
\((C+C^T)/2\) is positive definite. Similarly, there exists a
\(\gamma>0\) satisfying \(\rho(I-\gamma C)<1\) if and only if every
eigenvalue of \(C\) has positive real part. In either case, \(\gamma\)
must be sufficiently small. In particular, no \(\gamma>0\) yields a
convergent Neumann series when \(C\) is symmetric indefinite.
Alternative splittings and sampling constructions developed
for Monte Carlo linear solvers
\cite{benzi2017analysis,regenerative} provide possible
extensions, which are left for future work.}

The following theorem formalizes how the truncation and sampling errors should be
balanced against the analog hardware realization error.

\begin{theorem}[Noise-matched MCAI construction]
\label{thm:noise_matched_mcmc}
Let \(C\in\mathbb{R}^{n_k\times n_k}\), let \(\gamma>0\), and define
\(H=I-\gamma C\). Assume \(\|H\|_2\le \nu<1\), and let
\[
    S_\ell=\gamma\sum_{t=0}^{\ell}H^t
\]
be the truncated Neumann approximation. Let \(M_{\ell,n_c}^{\rm MC}\) be the
Monte Carlo estimator of \(S_\ell\) based on \(n_c\) independent walks per row,
and assume that, for some \(\Theta_\ell\) independent of \(n_c\),
\[
    \mathbb{E}\|M_{\ell,n_c}^{\rm MC}-S_\ell\|_F^2
    \le
    \frac{\Theta_\ell^2}{n_c}.
\]
Suppose the programmed analog realization is
\begin{equation}
    \label{eq:dhwdef}
    \widetilde M=M_{\ell,n_c}^{\rm MC}+\Xi,
    \qquad
    \delta_{\rm hw}:=\sqrt{\mathbb{E}\|\Xi\|_F^2},
\end{equation}
where \(\Xi\) is the fixed hardware programming perturbation after rescaling to
logical matrix coordinates. Then,
\begin{equation}
\label{eq:noise_matched_mcmc_bound}
    \mathbb{E}\|C^{-1}-\widetilde M\|_F
    \le
    \gamma\frac{\sqrt{n_k}\nu^{\ell+1}}{1-\nu}
    +
    \frac{\Theta_\ell}{\sqrt{n_c}}
    +
    \delta_{\rm hw}.
\end{equation}
\end{theorem}

\begin{proof}
Since \(H=I-\gamma C\), we have \(I-H=\gamma C\). The assumption
\(\|H\|_2<1\) implies that the Neumann series converges, and therefore
\[
    C^{-1}
    =
    \gamma(I-H)^{-1}
    =
    \gamma\sum_{t=0}^{\infty}H^t .
\]
Hence
\[
    C^{-1}-\widetilde M
    =
    (C^{-1}-S_\ell)
    +
    (S_\ell-M_{\ell,n_c}^{\rm MC})
    -
    \Xi .
\]
Taking expectations and applying the triangle inequality gives
\[
    \mathbb{E}\|C^{-1}-\widetilde M\|_F
    \le
    \|C^{-1}-S_\ell\|_F
    +
    \mathbb{E}\|S_\ell-M_{\ell,n_c}^{\rm MC}\|_F
    +
    \mathbb{E}\|\Xi\|_F .
\]
For the truncation term,
\[
    C^{-1}-S_\ell
    =
    \gamma\sum_{t=\ell+1}^{\infty}H^t .
\]
Thus,
\[
\begin{aligned}
    \|C^{-1}-S_\ell\|_F
    &\le
    \gamma\sum_{t=\ell+1}^{\infty}\|H^t\|_F  \le
    \gamma\sqrt{n_k}\sum_{t=\ell+1}^{\infty}\|H^t\|_2  \le
    \gamma\frac{\sqrt{n_k}\nu^{\ell+1}}{1-\nu}.
\end{aligned}
\]

For the sampling term, Jensen's inequality and the assumed sampling-error
bound give
\[
    \mathbb{E}\|S_\ell-M_{\ell,n_c}^{\rm MC}\|_F
    \le
    \sqrt{\mathbb{E}\|S_\ell-M_{\ell,n_c}^{\rm MC}\|_F^2}
    \le
    \frac{\Theta_\ell}{\sqrt{n_c}}.
\]
Similarly,
\[
    \mathbb{E}\|\Xi\|_F
    \le
    \sqrt{\mathbb{E}\|\Xi\|_F^2}
    =
    \delta_{\rm hw}.
\]
Combining the three bounds proves the result.
\end{proof}

In the scaled programming model, \(\Xi\) includes the factor \(\mu\) for pure
additive write noise, \(\Xi=\mu N^{(W)}\) in logical coordinates. Hence
\(\delta_{\rm hw}\) depends on both the hardware write-noise level and the
scaling required to program the constructed inverse block.
The three terms in \eqref{eq:noise_matched_mcmc_bound} are the
Neumann-series truncation error, the Monte Carlo sampling error, and the fixed
hardware programming error. The bound gives a noise-matching rule suggesting that 
the inverse should not be constructed to an accuracy far below
\(\delta_{\rm hw}\), because the programmed matrix already differs from the
constructed matrix by a perturbation of this size. Since \(\ell\) controls the
truncation error and \(n_c\) controls the sampling error, the two parameters
should be chosen jointly so that
\[
    \gamma\frac{\sqrt{n_k}\nu^{\ell+1}}{1-\nu}
    \lesssim
    \delta_{\rm hw},
    \qquad
    \frac{\Theta_\ell}{\sqrt{n_c}}
    \lesssim
    \delta_{\rm hw}.
\]
Increasing \(n_c\) cannot compensate for a walk length that is too short, and
increasing \(\ell\) may require increasing \(n_c\) if the sampling variance
grows with \(\ell\).

The first inequality gives a sufficient lower bound on the walk length. When
\(\delta_{\rm hw}(1-\nu)/(\gamma\sqrt{n_k})<1\), it is enough to take
\begin{equation}
\label{eq:mcmc_walk_length_rule}
    \ell
    \gtrsim
    \frac{
    \log\!\left(\delta_{\rm hw}(1-\nu)/(\gamma\sqrt{n_k})\right)
    }{\log(\nu)}
    -1 ,
\end{equation}
with integer rounding in practice. Once such an \(\ell\) is chosen, the second
inequality can be satisfied, provided \(\Theta_\ell<\infty\), by taking
\begin{equation}
\label{eq:mcmc_num_walks_rule}
    n_c
    \gtrsim
    \left(\frac{\Theta_\ell}{\delta_{\rm hw}}\right)^2 .
\end{equation}
If the right-hand side of \eqref{eq:mcmc_walk_length_rule} is negative, it means that
the truncation term is  below the hardware level even for \(\ell=0\). In
practice, \(\Theta_\ell\) can be estimated empirically from independent
samples, and the pair \((\ell,n_c)\) should be chosen subject to the available
setup budget.

\edit{In the MCAI experiments of Sections~\ref{sec:poisson_approx_inverse} and~\ref{sec:large_poisson},
each block \(C\) is SPD and
\(\gamma^{-1}=1.1\lambda_{\max}(C)\), so
$
    \|H\|_2=\rho(H)
    =1-\gamma\lambda_{\min}(C)<1$.
For fixed walk length \(\ell\), the finite-state estimator
has finite second moments because
\(P_{pq}>0\) whenever \(H_{pq}\neq0\).
Averaging \(n_c\) independent walks per row therefore gives
the \(1/n_c\) mean-square sampling-error scaling assumed
in Theorem~\ref{thm:noise_matched_mcmc}. This verifies the contraction and
sampling-error assumptions for these experiments.
}

\subsection{Damping, shifting, and nested block Jacobi}
\label{sec:damping_nesting}

The perturbation model in \eqref{eq:block_error_expansion} suggests two 
mechanisms to improve robustness. 
The first is to modify the local inverse action 
by \emph{damping and shifting},
so that
the preconditioned vector is not determined entirely by noisy 
analog application. For a diagonal block \(A_{kk}\), define
\[
    B_{kk}^{(\sigma)} := (A_{kk}+\sigma I)^{-1}
\]
as the inverse operator of the shifted diagonal block.  
We denote by
\[
    {\cal A}_{B_{kk}^{(\sigma)}}(r_k)
    =
    B_{kk}^{(\sigma)}r_k
    +
    \delta_{B_{kk}^{(\sigma)}}(r_k)
\]
the analog realization of applying \(B_{kk}^{(\sigma)}\) to \(r_k\), and
consider the damped block action
\begin{equation}
\label{eq:shifted_damped_block}
    z_k
    =
    {\cal A}_{B_{kk}^{(\sigma)}}(r_k)
    +
    \omega r_k .
\end{equation}
In exact arithmetic, this corresponds to the block preconditioner
\[
    M_{kk}^{(\sigma,\omega)}
    = B_{kk}^{(\sigma)}
    +
    \omega I.
\]
The shift \(\sigma\) regularizes the local inverse and may reduce the norm or
dynamic range of the programmed block. The identity term \(\omega r_k\) 
acts as
a reliable unpreconditioned search direction. This is useful
when the analog inverse action is strongly perturbed, since the preconditioned vector
retains part of the original input direction \(r_k\), instead of relying only
on the noisy analog inverse action. Thus, \(\omega\) can be viewed as a damping
or safeguarding parameter. If \(\omega\) is too large, however, the
preconditioner approaches a scaled identity map and loses the benefit of the
local inverse.

The second mechanism is to reduce the size of the matrices programmed onto the
analog hardware by \emph{nested preconditioning}. 
Rather than explicitly applying an approximate inverse of
\(A_{kk}\), we may compute \(y_k\approx A_{kk}^{-1}r_k\) by solving
\begin{equation}
\label{eq:nested_block_solve}
    A_{kk}y_k=r_k
\end{equation}
with inner FGMRES iterations, preconditioned by another level of an analog
block Jacobi approximate inverse built from smaller sub-blocks of \(A_{kk}\).
If the outer matrix is partitioned into \(p\) diagonal blocks and each diagonal
block is further partitioned into \(q\) sub-blocks, then the analog blocks have
dimension approximately \(n/(pq)\), instead of \(n/p\). Such a two-level partitioning is depicted in Figure~\ref{fig:matrixnest}; each linear system~\eqref{eq:nested_block_solve} is preconditioned by $\operatorname{blkdiag}(D_1^{(k)},\ldots,D_q^{(k)})$. In the common case
\(q=p\), the programmed block size is approximately \(n/p^2\). Notably, very small analog
crossbars may be inefficient because ADC/DAC overhead becomes large relative to the
amount of analog computation; however, very large tiles are more sensitive to wire,
current, and thermal non-idealities \cite{analogfrontier}. 
Nesting therefore
provides a practical way to keep the programmed analog blocks within a more
favorable size range while reducing analog exposure from analog noise and retaining much of the effect of a larger block
solve.

\edit{The mathematical block size need not equal the
physical crossbar size. For example, some existing AIMC
hardware uses fixed \(256\times256\) tiles
\cite{legallo2023aimcchip}.
An inverse block larger than a tile can be partitioned
into tile-sized submatrices, with partial MVM outputs
accumulated digitally \cite{rasch2023hardwareaware}.
This mapping retains the intended inverse action,
but introduces tile-level scaling, perturbations,
and communication that must be accounted for in
an implementation.
Our nested construction takes a different approach:
the larger block system is solved approximately by
inner FGMRES, preconditioned by smaller analog
sub-block inverses. Thus, \(p\) determines the outer
partition, while \(q\) can be chosen so that each
programmed sub-block, of approximate dimension
\(n/(pq)\), fits within a tile.}

\edit{For blocks smaller than a tile, padding requires
the unused input lines to be clamped or disabled and
the unused outputs to be discarded.
Within our MVM model, the unused input channels must
also be excluded from input-noise injection; otherwise,
nominally zero padded inputs can contribute to the
retained outputs. This masking is not implemented
in the simulator used here.
The experiments therefore assess the numerical effects
of changing block sizes, rather than a complete
mapping and scheduling strategy for fixed-size
physical tiles.}
\begin{figure}[t]
    \centering
    \includegraphics[width=0.72\linewidth]{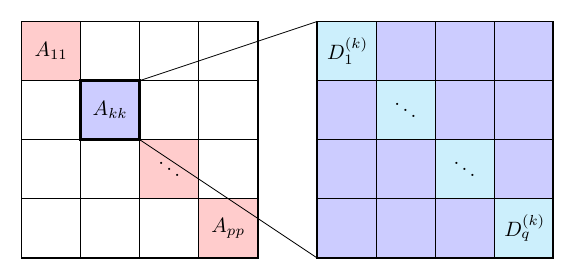}
    \caption{Nested block Jacobi preconditioning. Instead of applying an
    approximate inverse of \(A_{kk}\) directly, the block system
    \(A_{kk}y_k=r_k\) is solved approximately by an inner Krylov method
    preconditioned with smaller analog block Jacobi approximate inverses.}
    \label{fig:matrixnest}
\end{figure}

Shifting, damping, and nesting play complementary roles. Shifting and damping reduce sensitivity
to unreliable analog inverse actions, while nesting reduces the size of the
analog blocks that must be programmed and applied.

\section{Numerical Experiments}
\label{sec:experiments}

This section evaluates the analog-aware block Jacobi preconditioners developed
in Section~\ref{sec:alg}. The experiments connect the analog MVM perturbation
model in Section~\ref{sec:analog_noise_model} with observed FGMRES convergence
behavior. We first use a spectral diagnostic on a smaller Poisson problem to
illustrate how analog programming perturbations can change the realized
preconditioned operator. We then study FGMRES convergence of a
Poisson problem on a
\(16\times16\times16\) mesh, using exact block inverses to isolate
analog preconditioner application error and to examine block-size and damping
tradeoffs. We next compare SPAI and MCAI blocks where the MCAI parameter study
tests the saturation behavior suggested by the noise-matching principle in
Theorem~\ref{thm:noise_matched_mcmc}. Finally, we test nested block
preconditioning on nonsymmetric and larger Poisson problems to evaluate
whether reducing the analog block size preserves the convergence behavior of
the corresponding digital preconditioners.

The main performance measure is the number of FGMRES iterations required to
reach the prescribed relative residual norm tolerance. For comparisons with a
digital preconditioner, we also use the analog-to-digital iteration ratio
\[
    R_{\rm A/D}
    =
    \frac{\text{\texttt{\#} analog preconditioned iterations}}
         {\text{\texttt{\#} digital preconditioned iterations}} \ .
\]

\edit{The objective of these experiments is to determine
whether analog preconditioner applications retain useful
convergence behavior under the specified perturbation
model. The simulator is used to evaluate numerical
errors, and its execution time is not a prediction of
hybrid hardware runtime. Such a prediction additionally
requires a model of programming, conversion,
communication, and block scheduling on a specified
architecture.} \edit{Comparable outer iteration counts alone do not
establish runtime or energy savings, particularly for
nested preconditioning, where inner iterations introduce
additional work.}

All experiments were conducted in \textsc{Matlab} with 64-bit arithmetic on
all 8 cores of an AMD Ryzen AI 7 PRO 350 CPU. A \textsc{Matlab}
implementation of the publicly available analog hardware simulator was used
\cite{analogsim}. Unless otherwise stated, unrestarted FGMRES was run to
a relative residual norm tolerance of \(10^{-8}\). The analog simulator parameters
are listed in Table~\ref{tab:noisevalues}, \edit{where all noise is assumed to be mean-zero Gaussian}. Each plotted data point
represents the mean FGMRES iteration count over 25 right-hand sides \edit{with zero initial guess}. For
analog experiments, this corresponds to 5 independent analog preconditioner
realizations, each applied to 5 \edit{unit-norm} right-hand sides. Shaded regions denote plus-or-minus 
one standard deviation. Cases that did not converge within the
maximum allowed number of FGMRES iterations \edit{(400)} are indicated in the corresponding
figure captions. \edit{Unless otherwise stated, \(\omega = \sigma = 0\).} MCAI-specific parameters, such as \(n_c\) and \(\ell\), are stated in the experiments where they are used.

Unless otherwise stated, the SPAI experiments use an adaptive \cite{grote1997parallel} sparsity pattern
\(\mathcal{S}\) with a nonzero-entry budget
\[
    \mathtt{nnz}_{\rm SPAI}=40\,\mathtt{nnz}(A_{kk}),
    \qquad
    \mathtt{tol}_{\rm SPAI}=0.05 .
\]
Entries outside the prescribed sparsity pattern are not explicitly masked in
the analog simulator and are therefore subject to write noise. This convention
is relevant for interpreting the analog SPAI results where sparsity directly
reduces storage and MVM cost in digital arithmetic, whereas in the analog
setting it also interacts with programming noise 
and the physical realization of zeros.

\subsection{Spectral diagnostic for preconditioned operators}
\label{sec:precond_spectra}

Before studying FGMRES convergence, we include a spectral diagnostic on a
smaller problem for which all eigenvalues of the preconditioned operators can
be computed explicitly. We discretize the three-dimensional Poisson equation
on a \(12\times 12\times 12\) grid using the standard 7-point finite difference
stencil, yielding a matrix of dimension \(n=12^3\). This experiment illustrates
the perturbation discussion in Section~\ref{sec:inverse_construction}: even
when an approximate inverse is effective in exact arithmetic, its realized
analog version can change both the location and the qualitative structure of
the preconditioned spectrum.

Because the full analog preconditioner action is input-dependent, spectra of analog preconditioners must be interpreted
carefully. For this diagnostic only, we freeze one analog programming
realization and examine the eigenvalues of the corresponding fixed
right-preconditioned matrix \(A\widetilde M\), where \(\widetilde M\) denotes the
programmed preconditioner after write perturbation. 
Figure~\ref{fig:precondspectra} compares convex hulls of spectra for exact
block inverses and SPAI blocks, under both digital and analog realizations and for several choices
of the number of blocks \(p\). In the digital exact-inverse case, the block
Jacobi preconditioner is SPD, and \(AM\) is similar to
the symmetric matrix \(M^{1/2}AM^{1/2}\); hence its eigenvalues are real. At the 
same time, neither SPAI nor analog-write perturbations 
preserve this spectral property. Consequently, the
realized operator \(A\widetilde M\) may have its eigenvalues
move into the complex plane. In this particular realization, the analog exact-inverse case has some
eigenvalues with negative real part, showing that the realized
right-preconditioned operator can lose the positive spectral structure present
in the corresponding digital exact-inverse case. This diagnostic supports the main point that
preconditioner quality cannot be assessed only from the ideal matrix \(AM\); rather, hardware perturbations must also be taken into consideration.

\begin{figure}[t]
  \centering
  \includegraphics[width=0.72\linewidth]{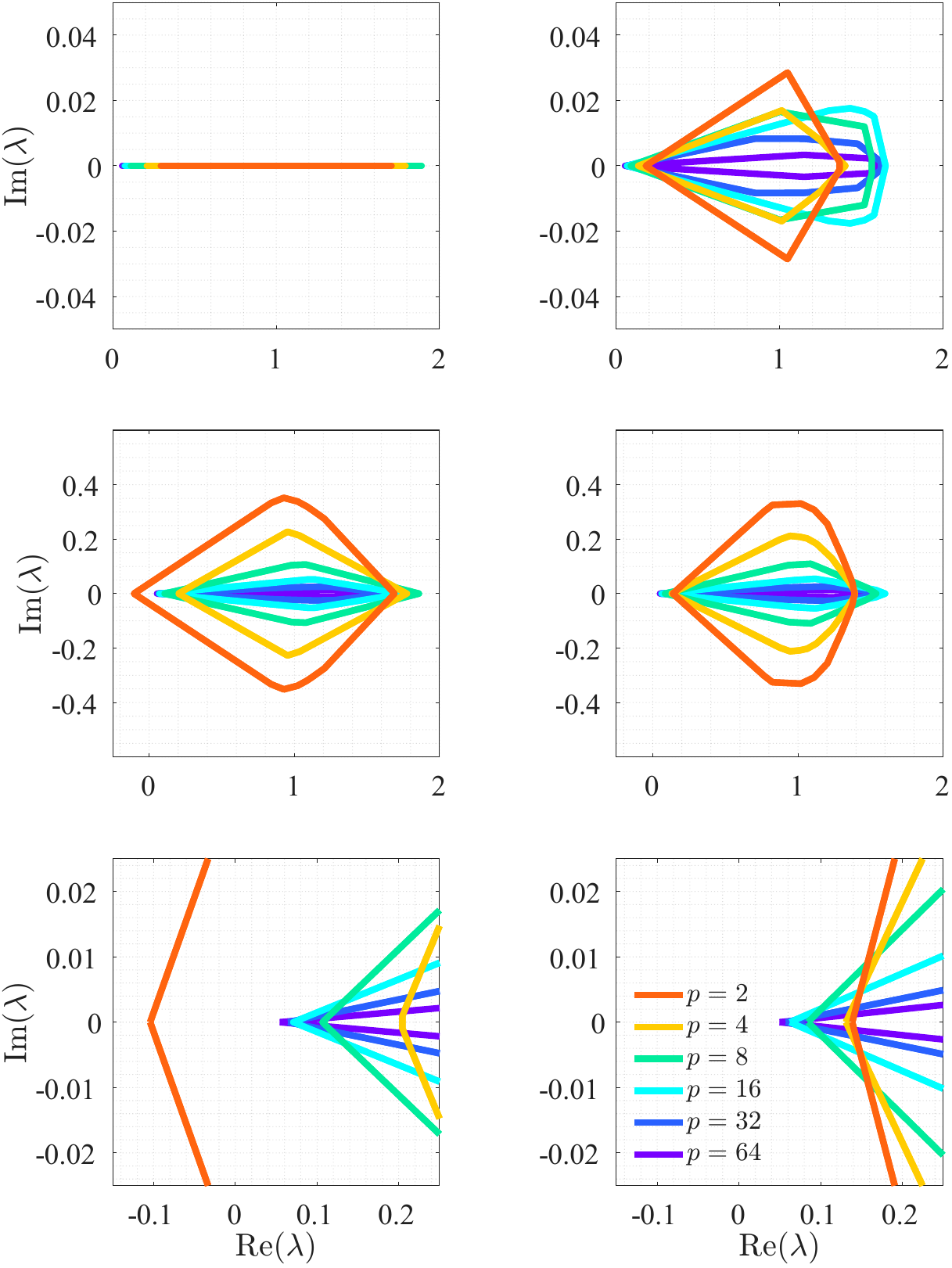}
  \caption{Convex hulls of spectra of right-preconditioned operators for the
  three-dimensional Poisson equation discretized on a
  \(12\times12\times12\) grid. Top row: digital preconditioners. Middle row:
  analog preconditioners with one frozen programming realization. Bottom row:
  zoomed view of the analog spectra near the origin. Left column: exact block inverses. Right
  column: SPAI blocks. The analog spectra show that write perturbations can
  shift the preconditioned spectrum into the complex plane.}
  \label{fig:precondspectra}
\end{figure}

\subsection{Convergence effects of block size and damping}
\label{sec:poisson3d_small}

We now turn from spectral diagnostics to FGMRES convergence. For these
experiments, we use the same three-dimensional Poisson model problem on a
larger \(16\times16\times16\) grid. The coefficient matrix is obtained from
the standard 7-point finite difference stencil and has dimension $n=16^3=4096$. This problem is small enough to allow multiple right-hand sides and independent
analog realizations, while still exhibiting the block-size and analog-noise
tradeoffs discussed in Section~\ref{sec:alg}.

\subsubsection{Block-size tradeoff with exact block inverses}
\label{sec:poisson_exact}

We begin with exact block inverses. For each diagonal block \(A_{kk}\), the
preconditioner block is computed in double precision as
\[
    M_{kk}=A_{kk}^{-1},
\]
and then applied either digitally or through the analog simulator. Since the
inverse blocks are exact at setup, this experiment isolates the effect of the
analog MVM realization where any difference between the digital and analog cases is
due to analog non-idealities in the preconditioner application.

Figure~\ref{fig:lapnoshift} compares digital and analog preconditioning as the
number of Jacobi blocks \(p\) varies. In exact digital arithmetic, increasing \(p\) decreases the block size and
weakens the block Jacobi preconditioner, since each diagonal block captures
less of the global coupling structure of \(A\). In contrast, for analog preconditioning, when \(p\) is small, the
blocks are large and the exact block inverses are strong, but the programmed
analog blocks are also more exposed to the perturbations described in
Section~\ref{sec:analog_noise_model}. When \(p\) is large, the analog exposure
per block is reduced, but the block Jacobi approximation is weaker. The
non-monotonic iteration behavior in Figure~\ref{fig:lapnoshift} is therefore
consistent with the error decomposition in \eqref{eq:block_error_expansion}.
\edit{Additional block-size studies on
\(10\times10\times10\) and \(14\times14\times14\)
Poisson grids gave \(p=4\) and \(p=6\), respectively,
as the tested values with the lowest mean analog
FGMRES iteration counts, compared with \(p=8\)
for the \(16\times16\times16\) grid.
The preferred number of blocks therefore increased
with problem size in these tests.
This trend is consistent with the block-size
dependence of the write-noise contribution in
\eqref{eq:write_noise_block_scaling}, but
the preferred \(p\) also depends on the block
scaling and exact-arithmetic preconditioning quality;
these results do not establish a general scaling law.}

\begin{figure}[t]
  \centering
  \includegraphics[width=0.88\linewidth]{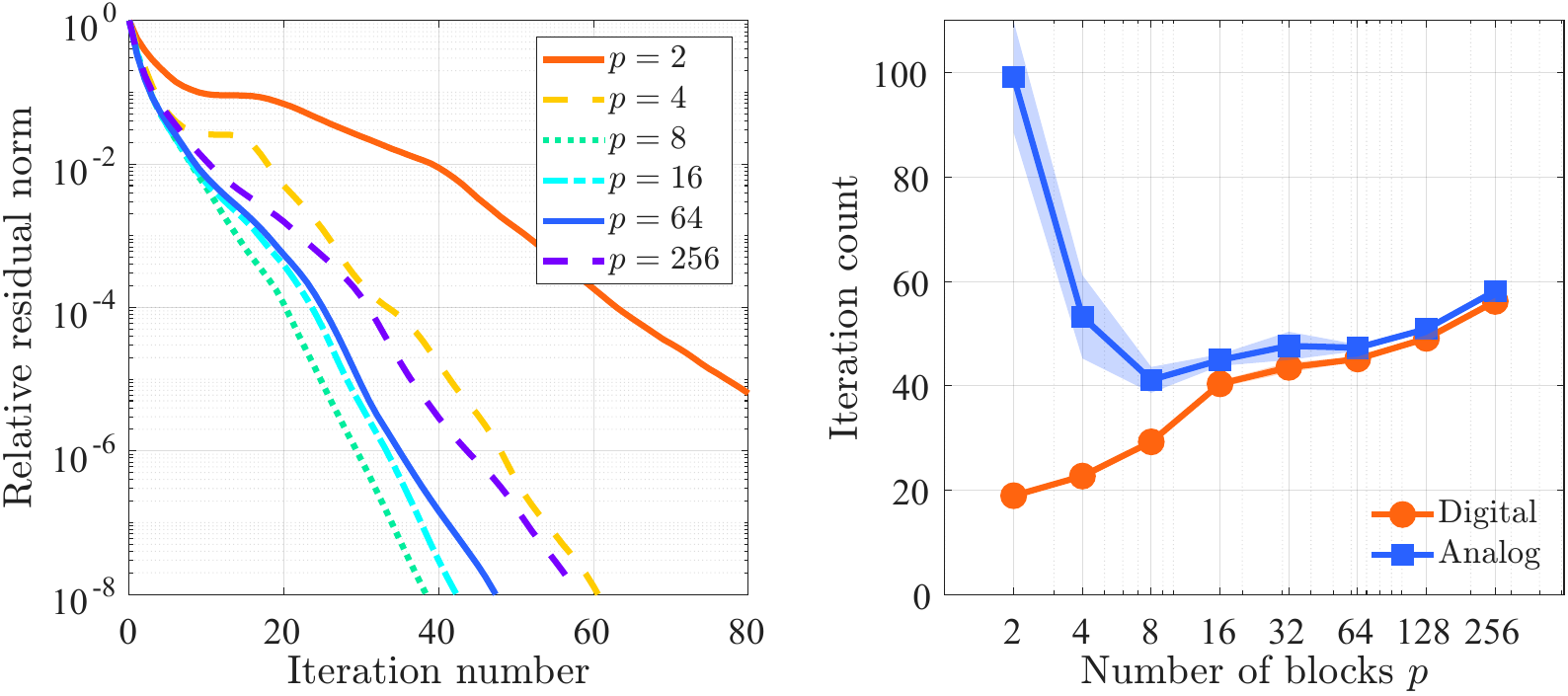}
  \caption{Digital and analog block Jacobi preconditioning with exact block
  inverses for the three-dimensional Poisson problem on a
  \(16\times16\times16\) grid. Left: analog convergence curves for a
  representative right-hand side. Right: mean FGMRES iteration count as the
  number of blocks \(p\) varies. Changing \(p\) changes both the
  exact-arithmetic preconditioner strength and the analog exposure of each
  programmed block, leading to a non-monotonic analog block-size tradeoff.}
  \label{fig:lapnoshift}
\end{figure}

We next test the shifted and damped block actions from
Section~\ref{sec:damping_nesting}. In exact arithmetic, the corresponding block
preconditioner is
\[
    M_{kk}^{(\sigma,\omega)}
    =
    (A_{kk}+\sigma I)^{-1}+\omega I .
\]
Figure~\ref{fig:lapshift} reports results for fixed \(p=8\). In the left
panel, the inner shift \(\sigma\) is varied with \(\omega=0\). In the right
panel, the outer damping parameter \(\omega\) is varied with \(\sigma=0\). The
inner shift gives little systematic improvement for this problem. By contrast,
a modest identity component \(\omega I\) reduces the mean analog iteration
count by roughly \(20\%\) relative to the undamped case. This agrees with the
interpretation in Section~\ref{sec:damping_nesting}, in which the \(\omega r_k\) term
provides a reliable component of the unpreconditioned search direction. If \(\omega\) is too large, however, the
preconditioner approaches a scaled identity map and the benefit of the block
inverse is lost.

\begin{figure}[t]
  \centering
  \includegraphics[width=0.88\linewidth]{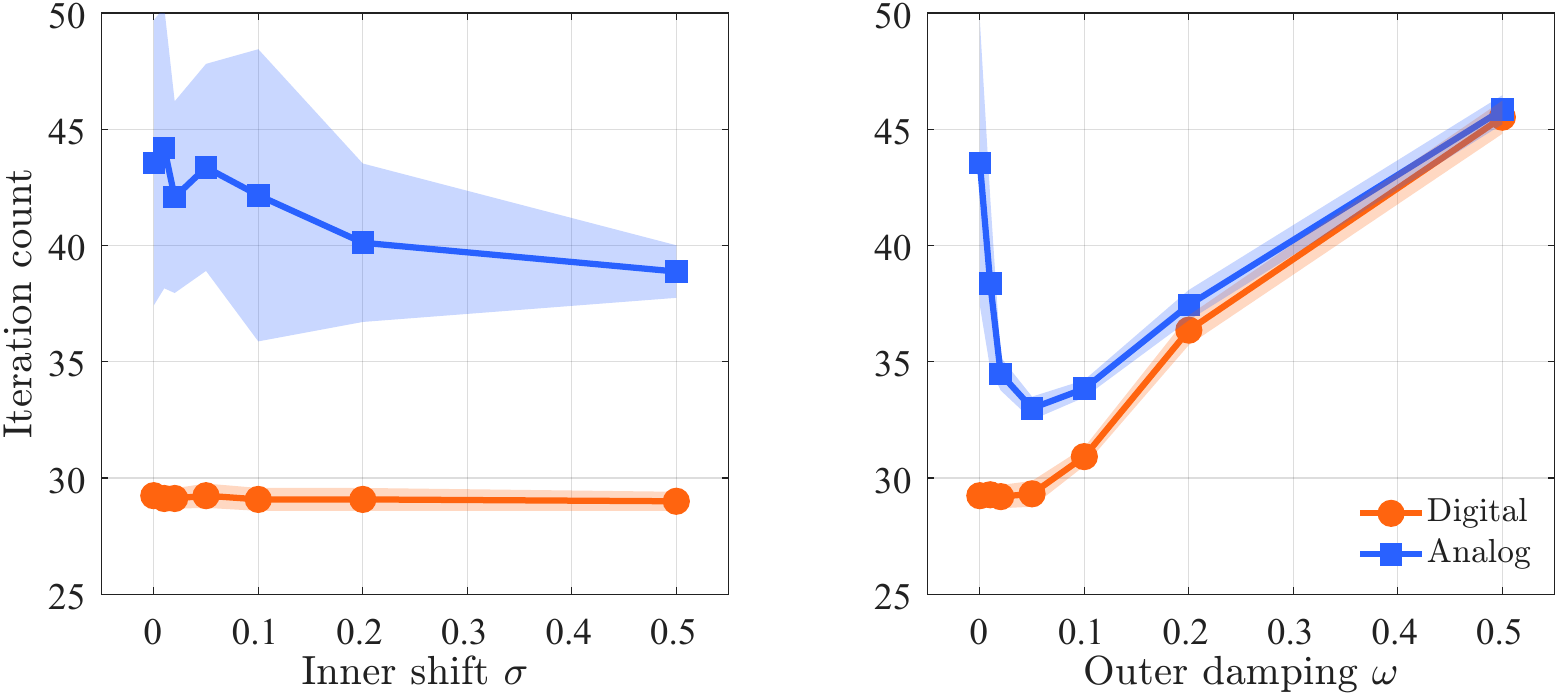}
  \caption{Analog preconditioning with exact block inverses for the three-dimensional Poisson problem on a
  \(16\times16\times16\) grid, for fixed \(p=8\).
  Left: the inner shift \(\sigma\) is varied with \(\omega=0\). Right: the
  outer damping parameter \(\omega\) is varied with \(\sigma=0\). The inner
  shift provides little systematic benefit, while a modest outer identity term
  improves robustness. Excessive damping weakens the preconditioner by moving
  it toward a scaled identity action.}
  \label{fig:lapshift}
\end{figure}

\subsubsection{SPAI and MCAI approximate inverse blocks}
\label{sec:poisson_approx_inverse}
We next replace the exact block inverses by approximate inverse constructions.

The first construction is a sparse approximate inverse (SPAI). For each
diagonal block \(A_{kk}\), we compute a sparse matrix
\(M_{kk}\approx A_{kk}^{-1}\) by minimizing the Frobenius-norm residual
\(\|I-A_{kk}M_{kk}\|_F\) subject to a sparsity constraint, as described in
Section~\ref{sec:inverse_construction}. Figure~\ref{fig:spai} shows the
resulting analog FGMRES iteration counts. The SPAI preconditioner exhibits the
same qualitative behavior observed for exact block inverses, i.e., the iteration
count is non-monotonic in the number of blocks and a modest damping parameter
improves robustness. This also showcases that minimizing the exact-arithmetic
residual \(\|I-A_{kk}M_{kk}\|_F\) alone is not sufficient to predict analog
performance, because the realized preconditioner also depends on the analog
perturbation terms in \eqref{eq:block_error_expansion}.

\begin{figure}[t]
  \centering
  \includegraphics[width=0.88\linewidth]{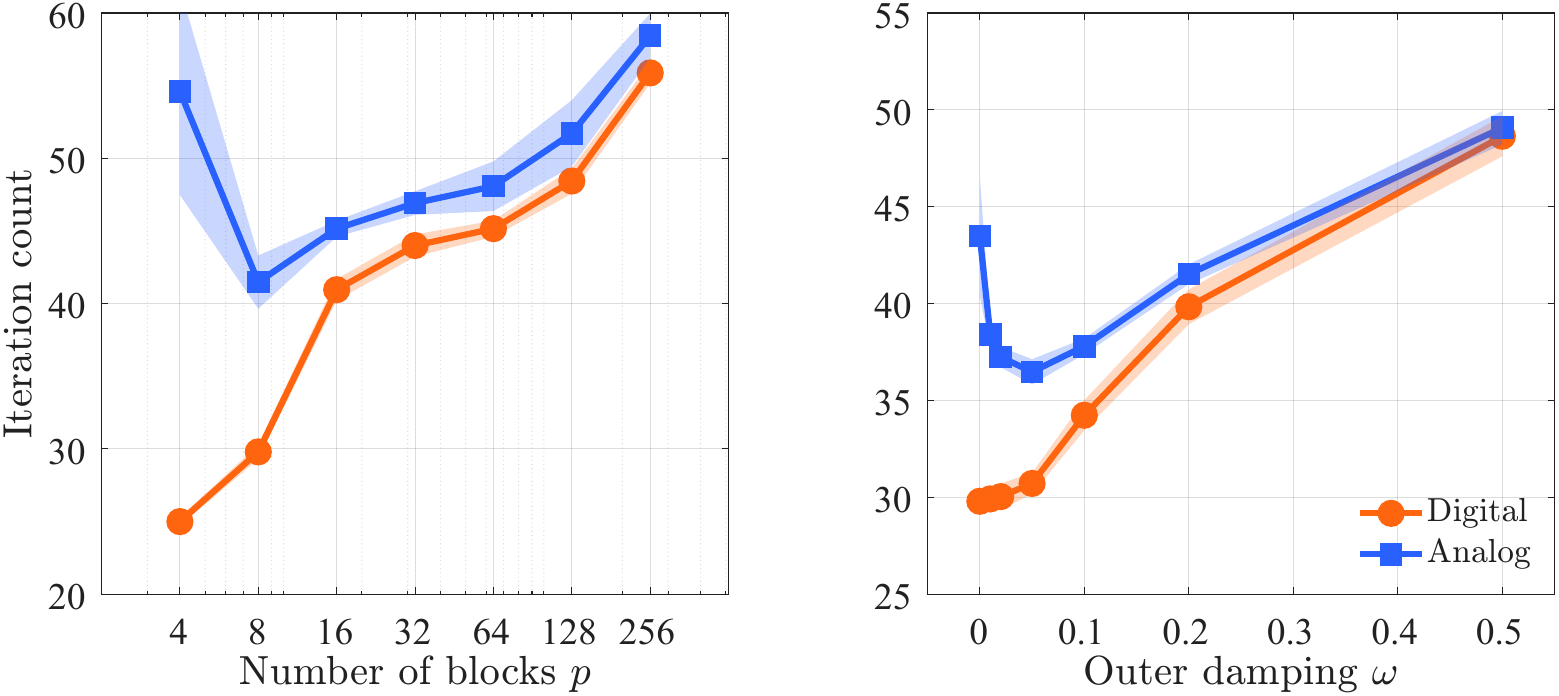}
  \caption{Analog block Jacobi preconditioning with SPAI blocks for the
  three-dimensional Poisson problem on a \(16\times16\times16\) grid. Left:
  mean FGMRES iteration count as the number of blocks varies. Right: damping
  study for fixed \(p=8\) with \(\sigma=0\). The results show the same
  analog-specific block-size tradeoff observed for exact inverses, together
  with a useful damping regime.}
  \label{fig:spai}
\end{figure}

The second construction is the Monte Carlo approximate inverse (MCAI)
preconditioner based on a Neumann series, described in Section~\ref{sec:mcmc}.
The MCAI parameters are the number of random walks \(n_c\) and
the walk length \(\ell\).

To compare MCAI construction accuracy with analog programming error, we define
the relative construction error
\[
    E_{\rm constr}(\ell,n_c)
    =
    \frac{
    \sqrt{\mathbb{E}\|M_{\ell,n_c}^{\rm MC}-C^{-1}\|_F^2}
    }
    {\|C^{-1}\|_F},
\]
and the relative analog programming error
\[
    E_{\rm hw}
    =
    \frac{\delta_{\rm hw}}{\|C^{-1}\|_F}.
\]
Here, \(M_{\ell,n_c}^{\rm MC}\) is the MCAI block constructed with parameters
\((\ell,n_c)\). In the experiments, \(\delta_{\rm hw}\) (defined in \eqref{eq:dhwdef}) is estimated from \edit{25} independent analog
write-noise realizations using the same programming, scaling, and output
rescaling convention. Because the programmed block is \(\mu^{-1} M_{\ell,n_c}^{\rm MC}\), the corresponding
logical write perturbation includes the factor \(\mu\). Thus, \(E_{\rm hw}\) is
estimated using the same conventions as the analog simulator.
We measure the relative position of the construction error to the hardware
error by
\[
    R_{\rm MC}
    =
    \frac{E_{\rm constr}(\ell,n_c)}{E_{\rm hw}}.
\]
Values \(R_{\rm MC}>1\) indicate an under-resolved MCAI block whose error is
larger than the analog programming error. Values \(R_{\rm MC}\lesssim 1\)
indicate that the MCAI construction error is at or below the analog hardware
level. Driving \(R_{\rm MC}\) far below one produces a more accurate inverse in
exact arithmetic, but the additional accuracy may not be useful after analog
programming.

Figure~\ref{fig:noise_matched_mcmc} makes this noise-matching behavior
explicit. The left panel reports \(R_{\rm MC}\) for a representative diagonal
block with \(p=8\). For small \(\ell\) and \(n_c\), \(R_{\rm MC}>1\), so the
approximate inverse is under-resolved. As \(\ell\) and \(n_c\) increase,
\(R_{\rm MC}\) decreases and enters the noise-matched regime
\(R_{\rm MC}\lesssim 1\). The right panel reports the corresponding
analog-to-digital iteration ratio \(R_{\rm A/D}\) for the full preconditioned
FGMRES solve. The iteration ratio decreases as the MCAI block becomes more
accurate, but then plateaus once the construction error is below the analog
programming error.\footnote{\edit{This observed plateau is consistent with the
noise-matching principle motivated by
Theorem~\ref{thm:noise_matched_mcmc}. The theorem bounds
inverse-construction and programming errors; it does not
establish an FGMRES convergence rate.}}

\begin{figure}[t]
  \centering
  \includegraphics[width=0.95\linewidth]{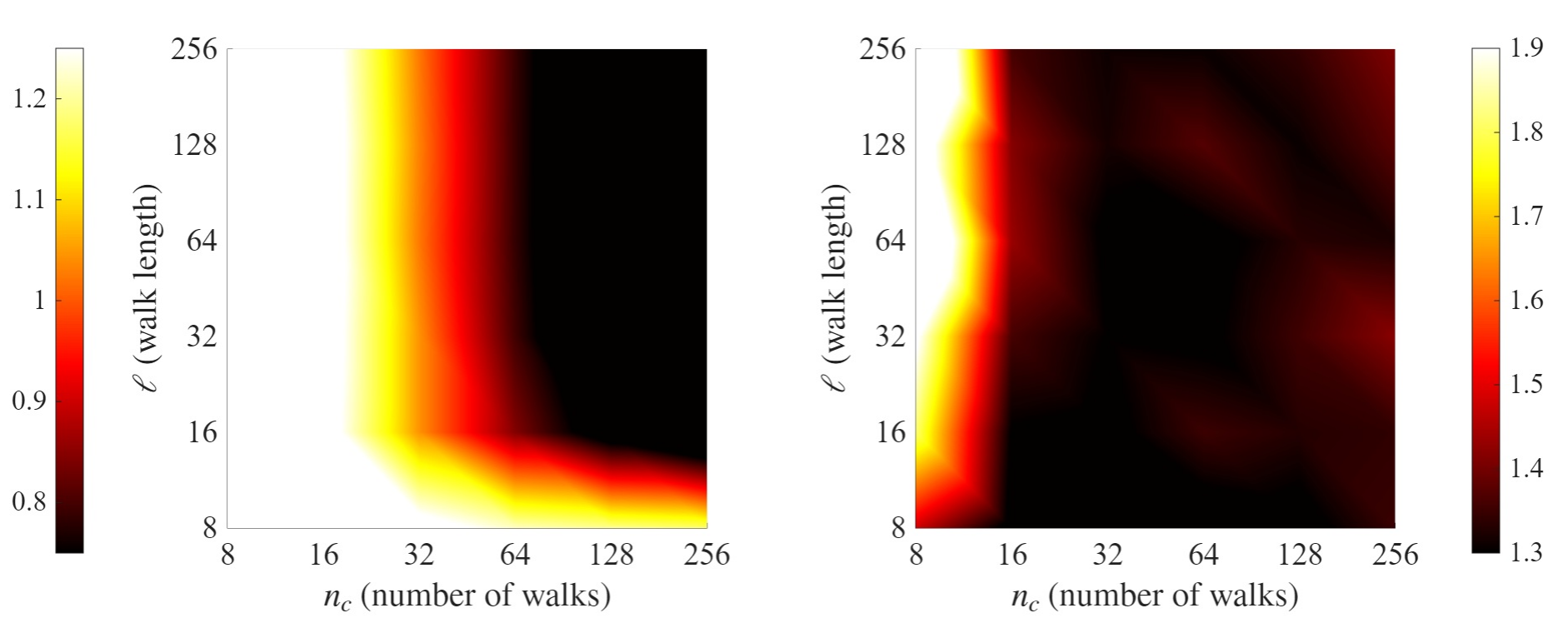}
  \caption{Effect of the MCAI parameters on FGMRES convergence for the
  \(16\times16\times16\) Poisson problem with fixed \(p=8\). The parameters are
  the number of random walks \(n_c\) and the walk length \(\ell\). Left:
  construction-error-to-hardware-error ratio \(R_{\rm MC}\) for a representative
  block. Values below one indicate that the MCAI construction error is no
  larger than the analog programming error. Right: analog-to-digital iteration
  ratio \(R_{\rm A/D}\). The iteration ratio plateaus as the MCAI construction
  is refined, consistent with the noise-matching principle of
  Theorem~\ref{thm:noise_matched_mcmc}.}
  \label{fig:noise_matched_mcmc}
\end{figure}

Table~\ref{tab:mcmc_regimes} summarizes three representative MCAI parameter
choices from Figure~\ref{fig:noise_matched_mcmc}. The under-resolved case has
\(R_{\rm MC}>1\) and requires substantially more FGMRES iterations in the analog setting. The
noise-matched and more-resolved cases both have \(R_{\rm MC}<1\), but the
more-resolved case provides little additional convergence benefit despite
requiring more setup work. 
This supports the practical rule suggested by
Theorem~\ref{thm:noise_matched_mcmc}, i.e.,  MCAI parameters should be chosen so that
the construction error is comparable to the analog programming error.

\begin{table}[t]
\centering
\caption{Representative MCAI parameter regimes for a Poisson block with
\(p=8\). The construction error \(E_{\rm constr}\) is measured relative to
\(\|C^{-1}\|_F\), and \(E_{\rm hw}\) is the corresponding relative analog
programming error. Their ratio is the construction-error-to-hardware-error ratio \(R_{\rm MC}\).}
\label{tab:mcmc_regimes}
\small
\begin{tabular}{c|c|c|c|c|c|c}
\textbf{Regime} &
\(\ell\) &
\(n_c\) &
\(E_{\rm constr}\) &
\(E_{\rm hw}\) &
\(R_{\rm MC}\) &
\textbf{Analog iterations (mean)} \\ \hline
Under-resolved & 16 & 16 & 0.311 & 0.240 & 1.3 & 55.8 \\ 
Noise-matched & 64 & 64 & 0.160 & 0.204 & 0.78 & 40.7 \\ 
More-resolved & 128 & 128 & 0.114 & 0.192 & 0.59 & 39.9 \\ \hline
\end{tabular}
\end{table}

For the remaining MCAI experiments, we use \(n_c=64\) and \(\ell=64\), which
places the construction in the noise-matched regime of
Figure~\ref{fig:noise_matched_mcmc}. Figure~\ref{fig:mcmc} then studies two
additional design choices for this fixed MCAI construction. The left panel
varies the number of Jacobi blocks \(p\), showing the same analog
block-size tradeoff observed for exact inverses and SPAI. The \(p=2\) analog
preconditioner did not converge within the maximum allowed number of FGMRES
iterations and is therefore omitted. The right panel fixes \(p=8\), sets
\(\sigma=0\), and varies the outer damping parameter \(\omega\). As in the
exact-inverse and SPAI experiments, a modest damping parameter improves
robustness, while excessive damping weakens the preconditioner.

\begin{figure}[t]
  \centering
  \includegraphics[width=0.88\linewidth]{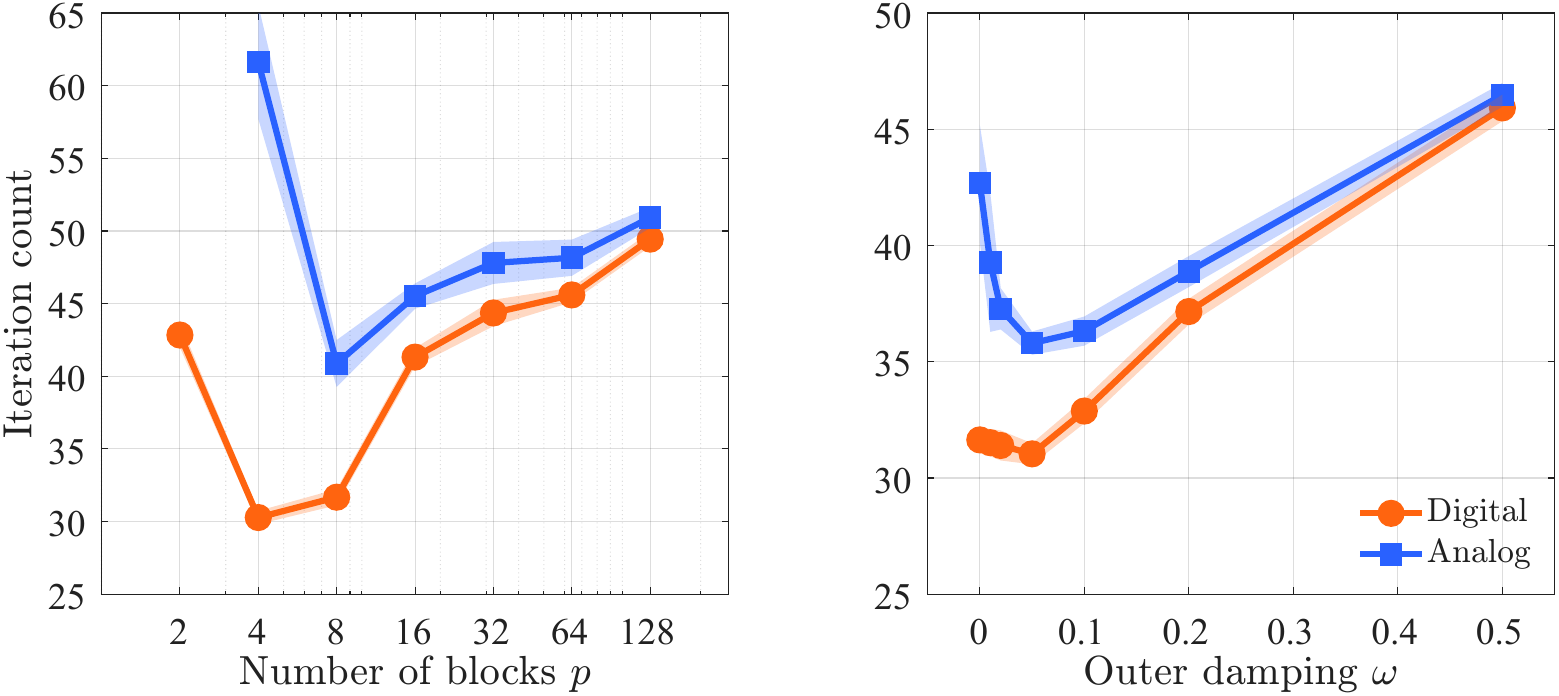}
  \caption{Analog block Jacobi preconditioning with MCAI blocks for the
  three-dimensional Poisson problem on a \(16\times16\times16\) grid. Left:
  mean FGMRES iteration count as the number of blocks \(p\) varies, using
  \(n_c=\ell=64\). The \(p=2\) analog case exceeded the maximum iteration count (400)
  and is omitted. Right: damping study for fixed \(p=8\) and \(\sigma=0\),
  varying the outer damping parameter \(\omega\).}
  \label{fig:mcmc}
\end{figure}

\edit{Figures~\ref{fig:spai} and~\ref{fig:mcmc} show similar
qualitative trends for SPAI and MCAI on this problem:
iteration counts depend non-monotonically on the number
of blocks, and modest damping improves robustness.
These comparisons do not match setup costs and do not
establish the superiority of either construction.
MCAI provides explicit controls over truncation and sampling
errors, which are the basis of the noise-matching analysis
developed here.}

\subsection{Analog tradeoffs for a nonsymmetric PDE problem}
\label{sec:cdr}

The preceding Poisson experiments involve an SPD coefficient matrix. We now test whether the same analog-aware behavior persists
for a nonsymmetric problem. More specifically, we consider the following three-dimensional
convection--diffusion--reaction equation
\[
    -h^2\nabla^2 u + 2h\,c\cdot \nabla u + \xi u = f,
\]
with
\[
    c=[-0.1,-1.2,-0.3]^T,\qquad \xi=-0.8,\qquad h=1/n_x .
\]
The Laplacian is discretized by the standard 7-point finite difference stencil,
and the convection term is discretized by centered differences without
upwinding. We use \(n_x=n_y=n_z=16\), so the matrix dimension is \(4096\).
In this subsection, the preconditioner blocks are exact inverses computed in
double precision before analog programming. Thus, as in
Section~\ref{sec:poisson_exact}, the experiment isolates the effect of analog
preconditioner application rather than inverse-construction error.

Figure~\ref{fig:cdr} shows that the nonsymmetric problem exhibits the same
qualitative analog-aware behavior as the Poisson problem. The \(p=2\) analog case exceeded the maximum iteration count. The remaining
data show that the analog iteration count is not monotone in the number of
blocks. More specifically, large Jacobi blocks give strong exact-arithmetic preconditioners but are
more sensitive to analog perturbations, while very small blocks reduce analog
exposure but weaken the block Jacobi approximation. A modest outer damping parameter \(\omega\) improves
robustness. These
results indicate that the block-size and damping tradeoffs predicted by the
perturbation model in \eqref{eq:block_error_expansion} are not restricted to
SPD matrices.

\edit{The cell P\'eclet number exceeds one only
in the \(y\) direction. Repeating this experiment with
first-order upwinding in that direction preserves the
non-monotonic dependence of the analog iteration count
on the number of blocks, although the best-performing
tested value shifts from \(p=4\) to \(p=8\).
Thus, the observed tradeoff in the number of blocks is not specific
to the centered-difference discretization in this test.}

\begin{figure}[t]
  \centering
  \includegraphics[width=0.88\linewidth]{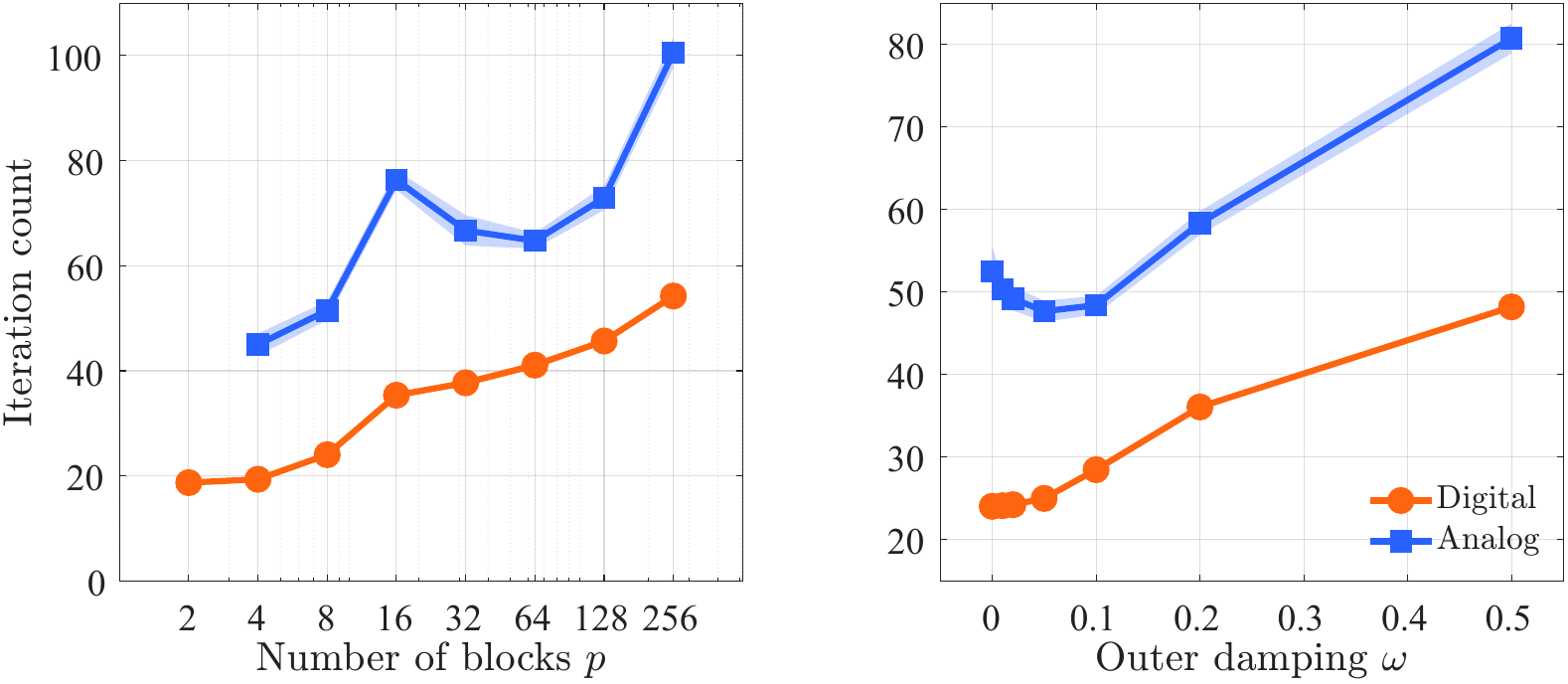}
  \caption{Analog block Jacobi preconditioning with exact block
  inverses for the nonsymmetric convection--diffusion--reaction problem on a
  \(16\times16\times16\) grid. Left:
  mean FGMRES iteration count as the number of blocks \(p\) varies. The \(p=2\) analog case exceeded the maximum iteration count (400) and is omitted. Right: damping study for fixed \(p=8\) and \(\sigma=0\),
  varying the outer damping parameter \(\omega\). The results show the same qualitative
  analog-aware behavior observed for Poisson: an intermediate block-size
  regime is most effective and a modest outer damping parameter \(\omega\) improves
  robustness.}
  \label{fig:cdr}
\end{figure}

\subsection{Nested preconditioning on the \texorpdfstring{$16^3$}{16x16x16} model problems}
\label{sec:nested_model_problems}

The previous experiments indicate that large programmed blocks can be sensitive
to analog perturbations. We therefore test the nested construction from
Section~\ref{sec:damping_nesting}, which reduces the size of the matrices
placed on analog arrays. We apply the same nested strategy to both
Poisson and
convection--diffusion--reaction \(16\times 16\times 16\) model problems.

In these experiments, the outer matrix is partitioned into \(p\) block Jacobi
blocks, and each diagonal block is further partitioned into \(p\) sub-blocks for
the inner preconditioner. Thus, the matrices actually programmed onto analog
hardware have dimension approximately \(n/p^2\), rather than \(n/p\). The local
systems
\[
    A_{kk}y_k=r_k
\]
are solved by an inner preconditioned FGMRES iteration with a relative residual norm tolerance
of \(10^{-2}\). To isolate the effect of nesting, the inverse of each smallest
sub-block is computed exactly before analog programming.

Figure~\ref{fig:lapnnested} shows that the nested analog preconditioner closely
tracks the corresponding  digital preconditioner in iteration count for
both model problems, indicating that nesting reduces analog exposure by using
smaller programmed blocks, while the inner iteration preserves much of the
convergence benefit of a larger block solve.

\begin{figure}[t]
  \centering
  \includegraphics[width=0.88\linewidth]{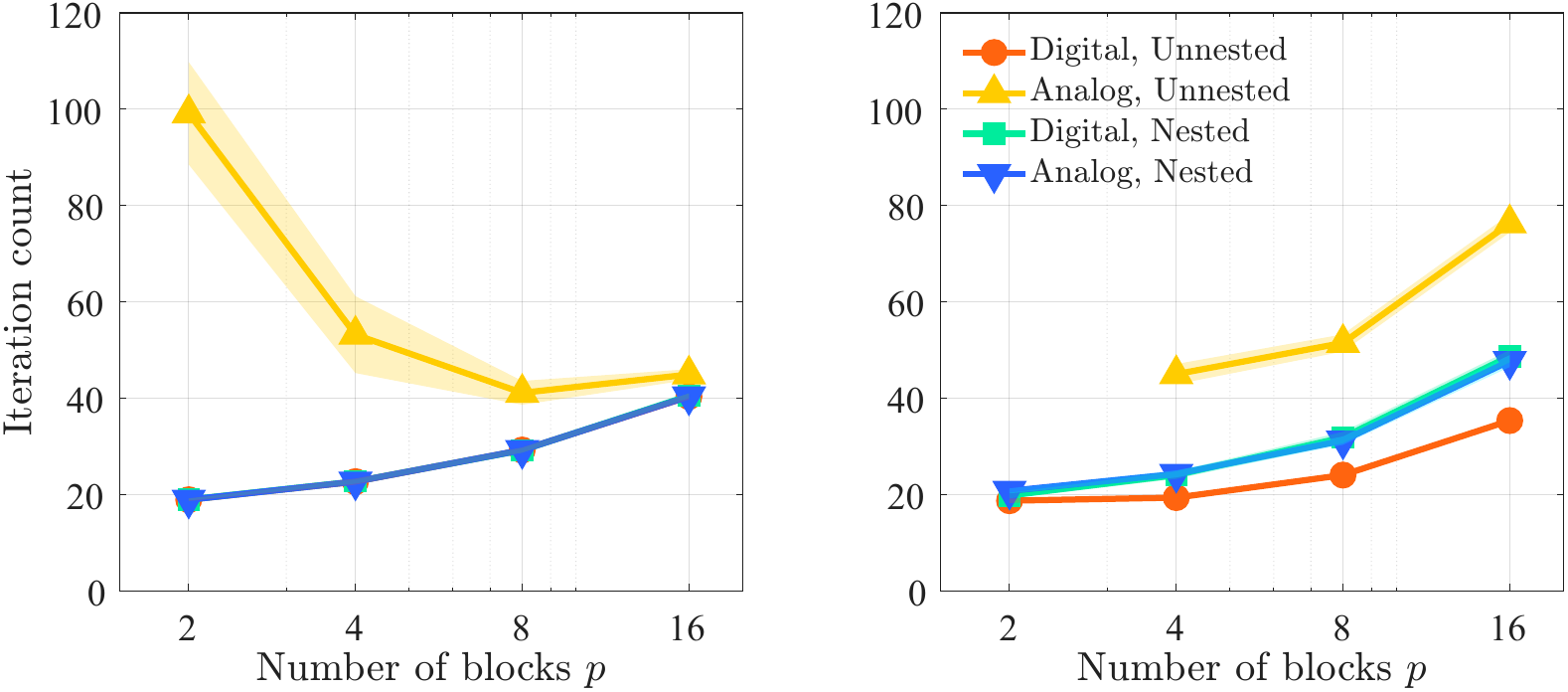}
  \caption{Nested block Jacobi preconditioning for the \(16\times16\times16\)
  model problems. Left: Poisson problem. Right: convection--diffusion--reaction
  problem. The outer and inner levels use the same number of blocks \(p\), so
  the analog sub-blocks have size approximately \(n/p^2\). The inner FGMRES
  tolerance is \(10^{-2}\). The \(p=2\) unnested analog case exceeded the maximum iteration count (400)
  and is omitted. In both cases, the nested analog preconditioner
  closely matches the nested digital preconditioner, indicating that reducing
  the analog block size can preserve the convergence benefits of larger block
  solves.}
  \label{fig:lapnnested}
\end{figure}

\subsection{Larger Poisson tests with nested analog MCAI preconditioning}
\label{sec:large_poisson}

The preceding experiments leverage nesting to reduce analog block
sizes. We now test this strategy on two larger Poisson problems. The
first is a three-dimensional (3D) Poisson problem on a \(32\times32\times32\) grid,
giving \(n=32768\). The second is a two-dimensional (2D) Poisson problem \edit{discretized using the standard 5-point finite difference stencil} on a
\(180\times180\) grid, giving \(n=32400\).

Table~\ref{tab:largepoisson} reports the mean and standard deviation of the
FGMRES iteration counts. We compare three preconditioned cases: (a) an unnested
digital block Jacobi preconditioner with exact block inverses, (b) a nested digital
block Jacobi preconditioner with exact sub-block inverses, and (c) a nested analog
block Jacobi preconditioner. The unnested preconditioner uses $p = 8$ blocks; for both nested cases, the outer and inner
levels both use \(p=q=8\) blocks. The analog sub-block preconditioners are MCAI
blocks with \(n_c=\ell=64\), together with damping parameter \(\omega=0.05\).

The nested analog preconditioner closely matches both
digital preconditioners in both larger tests. For the 3D Poisson problem, the
nested analog method requires \(38.4\) iterations on average, compared with
\(38.6\) for nested digital preconditioning and \(38.3\) for unnested digital
preconditioning. For the 2D Poisson problem, the nested analog method requires
\(93.8\) iterations on average, compared with \(94.0\) for nested digital
preconditioning and \(92.4\) for unnested digital preconditioning. Thus, at the
level of FGMRES iteration counts, the nested analog preconditioner retains
nearly the same convergence behavior as the corresponding digital block Jacobi
preconditioners.

\begin{table}[t]
\centering
\caption{FGMRES iteration counts for the larger Poisson experiments. The nested analog preconditioner uses MCAI sub-blocks with \(n_c=\ell=64\), damping
\(\omega=0.05\), and \(p=q=8\).}
\label{tab:largepoisson}
\small
\begin{tabular}{p{0.48\linewidth}|cc|cc}
\textbf{Preconditioner}
& \multicolumn{2}{c|}{\textbf{3D Poisson}}
& \multicolumn{2}{c}{\textbf{2D Poisson}} \\
\cline{2-5}
& mean & std. dev. & mean & std. dev. \\ \hline
Unnested, digital, \(p=8\)
& 38.3 & 1.0 & 92.4 & 1.7 \\ 
Nested, digital, \(p=q=8\)
& 38.6 & 0.7 & 94.0 & 0.8 \\ 
Nested, analog, MCAI, \(\omega=0.05\), \(p=q=8\)
& 38.4 & 0.9 & 93.8 & 1.0 \\ \hline
\end{tabular}
\end{table}

\section{Conclusions}
\label{sec:conc}

We studied block Jacobi approximate inverse preconditioning for FGMRES in a
hybrid analog--digital setting in which preconditioner construction and
Krylov iterations are performed digitally, whereas preconditioner applications
are executed on the analog device.
We developed various block Jacobi preconditioning schemes including
exact block inverses, sparse approximate inverses, and 
MCAI preconditioners.
The results show that analog preconditioner
design must balance exact-arithmetic inverse quality against robustness to
hardware-induced perturbations. In particular, stronger local inverses or
larger blocks are not always preferable once analog noise, quantization,
clipping, and dynamic range effects are taken into account. 
The analog-aware
viewpoint developed here motivates noise-matched MCAI construction, modest damping, and nested block solves. The numerical results
indicate that these choices can make analog preconditioner applications
closely track their digital counterparts in iteration count for the model
problems considered. \edit{The experiments use constant-coefficient model problems
and do not establish robustness for strongly heterogeneous
coefficients. Such problems require further investigation
of the interaction between coefficient contrast, inverse-block
scaling, and analog perturbations, as well as the possible
need for coarse-space or multilevel corrections.}

Future work includes overlapping block Jacobi variants,
systematic robustness studies across hardware noise levels, and extensions to
linear systems with multiple right-hand sides, where the cost of programming an
analog preconditioner could be amortized over many solves. Another direction is to study low-rank-based preconditioning within analog
computing architectures, including hierarchical matrix representations
\cite{smash,H1999,xu2025neuralapproximateinversepreconditioners} and
multilevel low-rank correction schemes \cite{mslr,gmslr}.

\section*{Acknowledgments}
The authors gratefully acknowledge the support of the U.S. Department of Energy, Office of Science under grant DE-AC52-07NA27344. Claude Opus 5.5 and GPT-6 Astra were used to assist with literature reviews and the final preparation of the manuscript. The authors assume responsibility for all content.

\bibliographystyle{siamplain}
\bibliography{references}

\end{document}

%% file: shared.tex
\usepackage{lipsum}
\usepackage{amsfonts}
\usepackage{graphicx}
\usepackage{epstopdf}
\usepackage{algorithmic}
\usepackage{nicefrac}
\usepackage{color}
\ifpdf
  \DeclareGraphicsExtensions{.eps,.pdf,.png,.jpg}
\else
  \DeclareGraphicsExtensions{.eps}
\fi
\definecolor{shikharnotecolor}{HTML}{7703FC}

\newcommand{\edit}[1]{\textcolor{black}{#1}}

\newsiamremark{remark}{Remark}
\newsiamremark{hypothesis}{Hypothesis}
\crefname{hypothesis}{Hypothesis}{Hypotheses}
\newsiamthm{claim}{Claim}
\newsiamremark{fact}{Fact}
\crefname{fact}{Fact}{Facts}

\headers{Hybrid Analog--Digital Approximate Inverse Preconditioning}{S. Shah, R. Li, T. Gokmen, V. Kalantzis, L. Horesh, and Y. Xi}

\title{Hybrid Analog--Digital Approximate Inverse Preconditioning for Krylov Methods}

\author{
Shikhar Shah\thanks{Department of Mathematics, Emory University, Atlanta, GA, USA
(\email{shikhar.shah@emory.edu}, \email{yxi26@emory.edu}).}
\and
Rui Peng Li\thanks{Center for Applied Scientific Computing, Lawrence Livermore National
Laboratory, Livermore, CA, USA
(\email{li50@llnl.gov}). The work of S. Shah, R. Li, and Y. Xi was performed under the auspices of
the U.S. Department of Energy by Lawrence Livermore National Laboratory under
Contract DE-AC52-07NA27344 (LLNL-JRNL-2020867).}
\and
Tayfun Gokmen\thanks{Thomas J. Watson Research Center, IBM, Yorktown Heights, NY, USA
(\email{tgokmen@us.ibm.com}, \email{vkal@ibm.com},
\email{lhoresh@us.ibm.com}).}
\and
Vassilis Kalantzis\footnotemark[3]
\and
Lior Horesh\footnotemark[3]
\and
Yuanzhe Xi\footnotemark[1]
}
\usepackage{amsopn}

\usepackage{xcolor}